\documentclass[12pt]{article}
 \pdfoutput=1
\usepackage{float}
\usepackage{graphicx}      
\usepackage{color}         
\usepackage{amsmath}              
\usepackage{amssymb}              
\usepackage{amsfonts}             
\usepackage{verbatim}             
\usepackage{latexsym}
\usepackage[normalem]{ulem} 
\usepackage[margin=1in]{geometry}

\def\qed{$\qquad \Box$}

\newtheorem{thm}{Theorem}[section]

\newtheorem{lem}[thm]{Lemma}
\newtheorem{rem}[thm]{Remark}

\def\para#1{\vskip 0.4\baselineskip\noindent{\bf #1}}
\newenvironment{proof}[1][Proof]{\textbf{#1.} }{\ \rule{0.5em}{0.5em}}

\newcommand{\thmref}[1]{Theorem~{\rm \ref{#1}}}
\newcommand{\lemref}[1]{Lemma~{\rm \ref{#1}}}

\makeatletter \@addtoreset{equation}{section}

\newcommand{\beq}[1]{\begin{equation} \label{#1}}
\newcommand{\eeq}{\end{equation}}
\newcommand{\bed}{\begin{displaymath}}
\newcommand{\eed}{\end{displaymath}}

\newcommand{\bea}{\bed\begin{array}{rl}}
\newcommand{\eea}{\end{array}\eed}

\newcommand{\barray}{\begin{array}{ll}}
\newcommand{\earray}{\end{array}}

\def\ds{\displaystyle}

\newcommand{\od}[1]{{\frac{d}{d #1}}}

\newcommand{\pd}[1]{{\frac{\partial}{\partial #1}}}

\newcommand{\pdd}[2]{{\frac{\partial #2}{\partial #1}}}

\numberwithin{equation}{section}

\newcommand{\gibson}{}

\title{Numerical Methods for Linear Diffusion Equations in the Presence of an Interface}
\author{V.~A.~Bokil\footnotemark[1], N.~L.~Gibson\footnotemark[2], S.~L.~Nguyen\footnotemark[3], E.~A.~Thomann\footnotemark[4] and E.~Waymire\footnotemark[5]\\
  Department of Mathematics\\
  Oregon State University\\
  Corvallis, OR 97331-4605
}

\date{}
\begin{document}

\maketitle

\footnotetext[1]{email: bokilv@math.oregonstate.edu}
\footnotetext[2]{email: gibsonn@math.oregonstate.edu}
\footnotetext[3]{email: sonluu.nguyen@upr.edu}
\footnotetext[4]{email: thomann@math.oregonstate.edu}
\footnotetext[5]{email: waymire@math.oregonstate.edu}

\noindent \textbf{Abstract:}
\medskip

\noindent We consider numerical methods for linear parabolic equations in one spatial
dimension having piecewise constant  diffusion coefficients defined by a one parameter family of interface conditions
at the discontinuity.  We construct immersed interface finite element methods for an alternative formulation of the original deterministic diffusion problem in which the interface condition is recast as a natural condition on the interfacial flux for which the given operator is self adjoint.  An Euler-Maruyama method is developed for the stochastic differential equation corresponding to the alternative divergence formulation of the equation having a discontinuous coefficient and a one-parameter family of interface conditions.  We then prove convergence estimates for the Euler scheme.   The main goal is to develop numerical schemes that can accommodate specification of any one of the possible interface conditions, and
to illustrate the implementation for each of the deterministic and stochastic formulations, respectively.  The issues pertaining to speed-ups of the numerical schemes are left to future work.
\medskip


\medskip


\clearpage

\section{Introduction}

The computational simulation of solutions to diffusion equations
in heterogeneous materials or landscapes
requires the use of highly efficient numerical methods which are
consistent, stable, and potentially have high orders of accuracy.
Discontinuities in parameters may decrease the overall accuracy of the
method if not handled appropriately.  The correct discretization at
the interface depends on the type of interface condition imposed by
the problem.

In diffusion models the transmission properties may be
coupled to physically discrete, discontinuous properties of the environment
such as river networks \cite{ignacio_etal_2010} or landscape topography
and meteorological conditions \cite{reiczigel_etal_2010,seno_koshiba_2005,lutscher_etal_2005,lutscher_etal_2006,Mayer_etal_2008}.
Diffusion equations provide one of the standard approaches to modeling population dynamics with dispersal in spatially patchy environments \cite{skellam1951random,cantrell2003spatial}.
There are a number of empirical studies which indicate that the dispersal behavior of individuals, such as species of insects including aphids, beetles and caterpillar, foraging honey bees as well as several species of butterflies, is influenced by boundaries (interfaces) between different types of habitats (patches) \cite{Schultz2001,turchin1989aggregation,ries2001butterfly,aizen2003bees}.

In recent work on interfacial effects \cite{Appuhamilage_AAP,Appuhamilage_WRR, waymire2011first}, the authors analyze the underlying stochastic process determined by the equation in divergence form and having a specific interfacial condition in the presence of discontinuities in diffusion coefficients across interfaces. The theory of Brownian motion applies to diffusion models in homogeneous media with constant coefficients \cite{einstein1956investigations}.  However, the discontinuity in the diffusion tensor at the interface between two media `skews' the basic particle motion. Incorporating bias in behavior/movement at an interface or patch boundary into diffusion models naturally leads to {\it Skew Brownian Motion (SBM)} \cite{ito1963brownian,walsh1,harrison1981skew,cantrell1999diffusion}, from which the underlying stochastic particle motions across the discontinuity, called $\alpha$-skew diffusion, can be constructed \cite{Appuhamilage_AAP}.
SBM assumes that particles (individuals) move according to ordinary diffusion until they encounter an interface, but at an interface the probability that a particle (individual) will move into the region on one side of the interface is different than the chance that it will move into the region on the other side; see \cite{cantrell1999diffusion} for the case of conservative interface conditions.

The basic idea to be
developed in the present article to deal with more general specifications of interface
conditions (than the conservative case) can be used in either deterministic or stochastic numerical framework.  Thus we have elected to present numerical approaches to both the deterministic and the stochastic equations in this single article.  Readers may be selective in this regard since the deterministic and stochastic methods are treated independently up to sharing common notation where possible.  The key idea that used for both approaches is a change of variables that transforms the given problem into one that involves a natural continuity of flux interface condition, rendering the problem self-adjoint, i.e., a form of {\it symmetrization}.

For the numerical simulation of diffusion equations with discontinuous coefficients involving special interface conditions we will develop immersed interface methods for the spatial discretization that have been recently formulated for elliptic problems \cite{leveque1994immersed,li2006immersed,li2003overview,li1998immersed,wiegmann1998immersed,mittal2005immersed, zhou2006high,zhao2004high}. For the time discretization we use implicit finite difference schemes such as Backward Euler and the Crank-Nicolson method. Here we present error estimates for the semi-discrete (continuous in time) problem with immersed finite element for the spatial discretization as well as error estimates for the fully discrete scheme.

In our previous work \cite{Appuhamilage_AAP} an equivalent
formulation in terms of solutions to stochastic differential
equations in which the effect of the interface is reflected in an
added drift rate involving the local time \cite{Appuhamilage_WRR} of
the process (the stochastic counterpart of the interface condition)
was developed. The numerical simulation of SDEs corresponding to
divergence form operators involving a discontinuous coefficient has
also been the subject of various articles in the recent past. In the
one-dimensional context, schemes based on random walks
\cite{etore2005random,lejay2011simulation,lejay2006scheme,lejay2009monte},
Euler methods \cite{martinez2006discretisation} (based on stochastic
Taylor expansions) and \cite{MartinezT12}, and exact simulation
methods \cite{etore2011exact} have been developed for the simulation
of the solution of such SDEs for the case of conservative
(self-adjoint) interface conditions.

The paper is organized as follows. We first introduce a natural one
parameter family of possible interface conditions coupled to a
diffusion problem, with discontinuous diffusion coefficient, in one
spatial dimension (Section \ref{sec:one}).  Motivating areas of application from the
engineering, ecological and biological sciences are briefly noted.
We then present a reformulation of the problem which naturally allows
the application of finite element methods, where the immersed
interface method is used to ensure that the basis functions satisfy
the reformulated interface condition (Section \ref{sec:ifem}).  We recall standard estimates
for the elliptic problem and then apply them to the case in
question.  We consider both backward Euler and Crank-Nicolson
methods for time discretization. We provide error estimates for the
fully discrete scheme with backward Euler time discetization and
verify rates with numerical examples.  Next, we introduce the
corresponding stochastic differential equation and develop the Euler-
Maruyama scheme for numerical solutions applicable to any one of the
interface conditions (Section \ref{sec:Stochastic}).  We prove convergence of the Euler-Maruyama
scheme under mild assumptions using the approach developed in
\cite{MartinezT12}. Finally, numerical simulations are provided that illustrate our theoretical results in Sections \ref{sec:numd} and \ref{sec:nums}.

\section{Diffusion with Discontinuous Coefficients}
\label{sec:one}
We consider the time dependent diffusion equation in one dimension with a piecewise discontinuous diffusion coefficient across an interface at $x = 0$ on which a one parameter family of interface conditions is prescribed.
We define the time interval $J= [0,T]$ and the domain $\Omega = \mathbb{R}$. The corresponding initial value problem on $\Omega \times J$ is given as
\begin{subequations}
\label{eq:diff}
\begin{align}
\label{eq:diff_pde}
\pdd{t}{u}(t,x) &=\pd{x}\left(\frac{D(x)}{2}\pdd{x}{u(t,x)}\right),
\quad \forall x\in\Omega, \ t \in J\setminus\{0\},\\
\label{eq:diff_sol}
u(t,0^+) &= u(t,0^-), \forall \ t \in J,\\
\label{eq:diff_lambda}
\lambda  \pdd{x}{u}(t,0^+) &= (1-\lambda) \pdd{x}{u}(t,0^-), \ \forall \ t \in J, \\
\label{eq:diff_init}
u(0,x) &= u_0(x), \forall \ x \in \Omega.
\end{align}
\end{subequations}
In model \eqref{eq:diff} the diffusion coefficient $D$ is piecewise defined by
\begin{equation}\label{D(x)}
D(x)=\begin{cases}
D^+&\mbox{if }x>0,\\
D^-&\mbox{if }x<0,
\end{cases}
\end{equation}
for some positive constants $D^+, D^-$. We assume
initial data  $u(0,x)=u_0(x)$ given for all $x \in \Omega$
in equation \eqref{eq:diff_init}. Continuity of the solution
$u(t,x)$ at the interface $x=0$ given in \eqref{eq:diff_sol}, as
well as a condition at $x=0$ given in \eqref{eq:diff_lambda} that
depends on a parameter $\lambda$ with $0<\lambda<1$, and involves
the derivative of the solution, specify the nature of the interface.
The choice of the value of $\lambda$ varies according to the
application,  and may be a function of $D^+$ and  $D^-$.

\begin{rem}
One may note that the extreme cases in which $\lambda = 0, 1$, respectively, correspond to
Neumann boundary conditions at the point of interface.  In particular, therefore the coefficients
are purely constant (smooth) on the corresponding half-line and amenable to standard approaches
to Neumann boundary value problems.  From this perspective there is no loss to restricting
considerations to $0 < \lambda < 1$.
\end{rem}

From the point of view of applications to environmental sciences, the cases of $\lambda =\lambda^*:={D^+\over D^+ + D^-}$ (continuity of flux), $\lambda= \lambda^\#: = 1/2$ (continuity of derivatives), and $\lambda = 0$, arise as solute transport interfaces \cite{Appuhamilage_AAP,appuhamillage2012interfacial}, upwelling of ocean current modeling \cite{RPM08}, and one-sided barrier (reflective) regions, respectively.
There are ecological species, example Fender's blue  butterfly, and aphids for which inter-facial effects are widely reported from experiments, but the precise interface condition is unknown from a mathematical perspective. e.g., see \cite{Schultz2001,turchin1989aggregation}. For the latter, the problem of determining $\lambda$ can also be treated as a statistical problem.

\subsection{Reformulated (Symmetrized) Model}
\label{reformulation}

In order to setup the problem for easy application of the deterministic and stochastic numerical methods, it is convenient to relate the parameter $\lambda$ in the interface condition \eqref{eq:diff_lambda} to one which appears in a reformulation of problem \eqref{eq:diff} written in self-adjoint form. We do this via multiplication of both sides of
the PDE in \eqref{eq:diff_pde} by a piecewise defined (positive) function
\begin{equation}
c(x)=\begin{cases}
c^+:=\lambda/ D^+&\mbox{if }x>0,\\
c^-:=(1-\lambda)/ D^-&\mbox{if }x<0.
\end{cases}
\end{equation}
The resulting PDE can be written
\begin{equation}\label{eq:diffusion2}
c(x)\pdd{t}{u} = \pd{x}\left(\kappa(x)\pdd{x}{u}\right),
\quad \forall x\in\Omega, t \in J\setminus \{0\},
\end{equation}
where the positive function $\kappa$ is defined as
\begin{equation}
\kappa(x) = c(x)\frac{D(x)}{2}
           =\begin{cases}
\kappa^+:=\frac{\lambda}{2}&\mbox{if }x>0,\\
\kappa^-:=\frac{(1-\lambda)}{2}&\mbox{if }x<0.
\end{cases}
\end{equation}
Thus, the interface condition \eqref{eq:diff_lambda} may be interpreted as
\begin{equation}\label{eq:interface2}
\left[\kappa \pdd{x}{u}\right]:= \kappa^+\pdd{x}{u}(t,0^+)-\kappa^-\pdd{x}{u}(t,0^-) = 0,
\end{equation}
i.e., the {\em jump} across the interface of $\kappa\pdd{x}{u}$ at $x=0$, denoted as$\left[\kappa \pdd{x}{u}\right]$,  is zero. Thus, problem \eqref{eq:diff} can be reformulated to have an interface condition that resembles a natural flux condition (conservative) which is more easily amenable to numerical discretization. The reformulated version of problem \eqref{eq:diff} on $\Omega = \mathbb{R}$ can be stated as
\begin{subequations}
\label{eq:problem1}
\begin{align}
\label{eq:diff_reform_pde}
c(x)\pdd{t}{u}(t,x) &=\pd{x}\left(\kappa(x)\pdd{x}{u(t,x)}\right),
\quad \forall x\in\Omega, \ t \in J\setminus \{0\},\\
\label{eq:jump1}
\left[u\right]&:=u(t,0^+)-u(t,0^-) = 0, \ \forall \ t \in J,\\
\label{eq:diff_reform_lambda}
\left[\kappa \pdd{x}{u}\right]&:=\kappa^+\pdd{x}{u}(t,0^+)-\kappa^-\pdd{x}{u}(t,0^-)=0, \ \forall \ t \in J,\\
u(0,x) &= u_0(x), \forall x \in \Omega.
\end{align}
\end{subequations}
\begin{rem}
We note that $c$
plays the role of specific heat capacity times mass density of the
material, and $\kappa$ is a thermal conductivity, in the context
of heat flow.  We observe that for the special
case of $\lambda=\lambda^*:=D^+/(D^++D^-)$ we
have that $c(x)\equiv \mbox{\em constant}$.
\end{rem}

\section{The Immersed Finite Element Method (IFEM)}
\label{sec:ifem}

To construct a discrete solution of the problem \eqref{eq:problem1} and to generate numerical simulations we will need to consider problem \eqref{eq:diff} and hence problem \eqref{eq:problem1} on a finite interval. Thus, in this section we will formulate problem \eqref{eq:problem1} on the domain $\Omega :=(-L,L)$ for $L > 0, L \in \mathbb{R}$. In order for problem \eqref{eq:diff} and problem \eqref{eq:problem1} to be well-posed on $\Omega$ we will impose the boundary conditions
\begin{equation}
\label{eq:bc}
u(t,-L) = u(t,L) = 0, \ \forall \ t \in J,
\end{equation}
on the boundary of $\Omega$.

There are several numerical approaches available for the spatial discretization of parabolic interface problems like problem \eqref{eq:problem1} along with \eqref{eq:bc}. These include domain embedding methods like the fictitious domain method \cite{Glowinski_FDNS94,yu2006fictitious}, implicit derivative matching methods \cite{zhou2006high}, immersed boundary methods \cite{peskin2003immersed} and immersed interface methods based on either finite differences \cite{li2003overview,li2006immersed} or finite element methods \cite{li1998immersed}. In this paper we consider the immersed finite element method (IFEM), which is a numerical techinique based on the finite element method (FEM) for spatially discretizing problem \eqref{eq:problem1} along with \eqref{eq:bc}. Like FEM the IFEM is based on a variational formulation of the initial boundary value problem \eqref{eq:problem1} along with \eqref{eq:bc}. However, unlike the FEM the spatial mesh of the IFEM can be constructed independently of the interface. Also, unlike the FEM, some of the basis functions in the IFEM depend on the interface location at $x=0$ and the interface jump conditions \eqref{eq:jump1} and \eqref{eq:diff_reform_lambda}. We refer the reader to \cite{li2006immersed} for further details.

\subsection{Functional Spaces and the Variational Formulation}
We define the sub-domains $\Omega_1 = (-L, 0)$ and $\Omega_2 = (0, L)$, so that $\Omega = \Omega_1\cup \Omega_2\cup\{0\}$,
For $m>0$ and $1\leq p\leq \infty$, $H^{m}(\Omega)$ is the Sobolev space of order $m$ with norm $||\cdot||_{H^m(\Omega)}$ and for $m=0$, $H^0(\Omega) \equiv L^2(\Omega)$ with norm $||\cdot||_{L^2(\Omega)}$. For $k \geq 0$ we define the functional spaces
\begin{align}
H^{k,0}(\Omega) &= \{v \in L^2(\Omega) \ | \ v \in H^k(\Omega_1)\cap H^k(\Omega_2), [v] = 0, [\kappa \pdd{x}{v}] = 0\},\\
H_0^{k,0}(\Omega) &= \{v \in H^{k,0}(\Omega) \ | \ v(-L) = v(L) = 0\},
\end{align}
along with the norm
\begin{equation}
\label{eq:brokennorm}
||v||^2_{H^{k,0}(\Omega)} := \left(||v||^2_{H^{k}(\Omega_1)}+||v||^2_{H^{k}(\Omega_2)}\right),
\end{equation}
where $H^k(\Omega)$ and $H^k_0(\Omega)$ are the usual Sobolev spaces for $k\geq 1$. On $H^k(\Omega)$ we also define the seminorm  $|v|_{H^k(\Omega)} :=||\frac{\partial^kv}{\partial x^k}||_{L^2(\Omega)}$ and recall that for functions in $H^1_0(\Omega)$, the norms $||v||_{H^1(\Omega)}$ and $|v|_{H^1(\Omega)}$
are equivalent due to Friedrichs' lemma \cite{thomee2006galerkin}. The (semi) norms $|v|_{H^{k,0}(\Omega)}$ are defined in a similar manner to \eqref{eq:brokennorm}.
For a normed vector space $X$ and for $1\leq p\leq \infty$ we define the Lebesgue space $L^p(0,T;X)$ to be the space of all $X$ valued functions $v :J \rightarrow X$ for which $t\rightarrow ||v(t)||_X$ is in the Banach space $L^p(J)$ equipped with the norm
\begin{equation}
\label{eq:timenorm}
||v||_{L^p(X)} = \begin{cases}
\left(\int_0^T||v(t)||^p_X dt\right)^{\frac{1}{p}}, & \text{if} \ p < \infty\\
{\underset{t \in J}{\mathrm{ess \ sup}}} ||v(t)||_X, & \text{if} \ p = \infty.
\end{cases}
\end{equation}
On the space $L^p(0,T;H^{2,0}(\Omega)\cap H_0^1(\Omega))$ we will also define the (semi) norm $|v|_{L^p(H^{2,0}(\Omega))}$  in a similar manner to \eqref{eq:timenorm} using the seminorm  $|v|_{H^{2,0}(\Omega)}$ on $H^{2,0}(\Omega)\cap H_0^1(\Omega))$.
We define the operator $\mathcal{L}$ as
\begin{equation}
\begin{split}
&\mathcal{L}: C^0([0,1];H_0^{1,0}(\Omega))\times H_0^{1,0}(\Omega) \rightarrow C^0([0,1]; \mathbb{R}); \\
&\mathcal{L}(u,v) = \int_\Omega \kappa(x)\pdd{x}{u(t)}\pdd{x}{v} dx = \int_{\Omega^-} \kappa^-\pdd{x}{u(t)}\pdd{x}{v} dx+\int_{\Omega^+} \kappa^+\pdd{x}{u(t)}\pdd{x}{v} dx.
\end{split}
\end{equation}
The variational formulation corresponding to problem \eqref{eq:problem1} along with \eqref{eq:bc} is: \\
\noindent Find $u:J\longrightarrow H_0^{1,0}(\Omega)$ such that
\begin{equation}
\label{eq:problem2}
\int_\Omega c(x) \pdd{t}{u(t)} v dx +\mathcal{L}(u,v) = 0, \ \forall v \in H_0^{1,0}(\Omega).
\end{equation}

\subsection{Spatial Discretization Using Immersed Finite Elements}

We partition $\overline{\Omega} = [-L,L]$ using a uniform mesh
\[ -L = x_0 < x_1 < x_2 < \ldots < x_N = L.\]
We define the mesh step size, $h_i := x_i-x_{i-1}$, to be a constant $h$ for all subintervals in the partition $i = 1,2,\ldots , N$.
The discrete mesh is then denoted $\tau_h = {\underset{\hspace*{-15pt}i=0}{\overset{\hspace*{-15pt}N-1}{\bigcup E_i}}}$ with $E_i = [x_i, x_{i+1}]$.

At every node $x_i$ we define basis functions $\phi_j(x)$ as
\begin{equation}
\phi_j(x_i) = \begin{cases} 1, & i = j, \\
  0, & i \neq j,
\end{cases}
\end{equation}
satisfying the interface conditions $[\phi_j] = 0$ and $[\kappa\phi_j'] = 0$. We consider the linear IFE space
\begin{equation}
S_h^1(\Omega) = \text{span}\{\phi_j\}_{j=1}^{N-1}.
\end{equation}
Since $\forall j , \phi_j \in H^{1,0}(\Omega)$, we have that $S_h^1(\Omega) \subset H^{1,0}(\Omega)$. Let $x_j < 0 < x_{j+1}$, for some $j$, then the basis functions $\phi_j$ and $\phi_{j+1}$ are the only ones that need to be modified to satisfy the flux jump condition. This modification can be made using the method of undetermined coefficients as is done in \cite{li1998immersed}. We refer the reader to \cite{lin2007error,li1998immersed} for the construction of the IFE basis functions $\phi_j$.

\subsubsection{Interpolation Functions and Error Estimates}
To derive the error estimates for the time dependent problem \eqref{eq:problem1} along with \eqref{eq:bc}, we will use the error analysis derived for the corresponding stationary problem in \cite{li1998immersed,lin2007error} which is outlined below. Consider the stationary problem :
\begin{subequations}
\label{eq:elliptic}
\begin{align}
-\pd{x}\left(\kappa(x)\pdd{x}{u(t,x)}\right) & = f(x),
\ \forall x\in\Omega,\\
[u] :=u(0^+)-u(0^-) &= 0,\\
\left[\kappa\pdd{x}{u}\right]:=\kappa^+\pdd{x}{u}(0^+)-\kappa^-\pdd{x}{u}(0^-)&=0,\\
u(-L) = u(L) & = 0,
\end{align}
\end{subequations}
with $f \in L^2(\Omega)$.
We define the linear functional $\ell$ as
\begin{equation}
\ell:C^0([0,1];H_0^{1,0}(\Omega)) \rightarrow C^0([0,1]; \mathbb{R}); \
\ell(u)  = \int_\Omega fu dx.
\end{equation}
The weak form of problem \eqref{eq:elliptic} is:\\
Find $u \in H_0^{1,0}(\Omega)$ such that
\begin{equation}
\label{eq:ellipticweak}
\mathcal{L}(u,v) = (f,v), \forall \ v \in H_0^{1,0}(\Omega).
\end{equation}

The discrete variational problem using immersed finite elements is:\\
Find $u_h \in S_{h,0}^1:=S_h^1(\Omega)\cap H_0^{1,0}(\Omega)$ such that
\begin{equation}
\label{eq:ellipticdiscrete}
\mathcal{L}(u_h,v_h) = (f,v_h), \forall \ v_h \in S_{h,0}^1(\Omega).
\end{equation}

Based on the estimates for the interpolants error estimates for the IFE solutions are derived in \cite{li1998immersed,lin2007error}.
\begin{thm}[Theorem 4 from \cite{lin2007error}]
\label{thm:3}
Let $u \in H^{2,0}(\Omega)\cap H_0^{1,0}(\Omega)$ and $u_h \in S_{h,0}^1(\Omega)\subset H^{1,0}_0(\Omega)$ be solutions to \eqref{eq:ellipticweak} and \eqref{eq:ellipticdiscrete}, respectively. Then $\exists$ a positive constant $C$ independent of $u$ and $h$ such that
\begin{equation}
||u-u_h||_{H^{0,0}(\Omega)}+h||u-u_h||_{H^{1,0}(\Omega)} \leq C\rho h^2|u|_{H^{2,0}(\Omega)},
\end{equation}
where $\rho := \mathrm{max}\{\ds\frac{\kappa^-}{\kappa^+},\ds\frac{\kappa^+}{\kappa^-}\}$.
\end{thm}

\begin{thm}[Theorem 3.4 from \cite{li1998immersed}]
Let $u \in H^{2,0}(\Omega) \cap H_0^{1,0}(\Omega)$ and $u_h \in S_{h,0}^1(\Omega)\subset H^{1,0}_0(\Omega)$ be solutions to \eqref{eq:ellipticweak} and \eqref{eq:ellipticdiscrete}, respectively. Then $\exists$ a positive constant $\bar{C}$ independent of $u$ and $h$ such that
\begin{equation}
||u-u_h||_{L^{\infty}(\Omega)} \leq \bar{C}h^2||\frac{\partial^2u}{\partial x^2}||_{L^{\infty}(\Omega)},
\end{equation}
where $\bar{C} = \ds\frac{2\mathrm{max}\{1,\rho\}}{\mathrm{min}\{1,\rho\}}+\ds\frac{3}{2}$, and $\rho$ as defined in Theorem \ref{thm:3}.
\end{thm}

\subsection{Semi-Discrete Schemes: The Continuous Time Galerkin Immersed Finite Element Problem}
We now study the convergence properties of a semi-discrete scheme applied to the time dependent problem \eqref{eq:problem1} obtained by spatially discretizing the problem using IFEM. Denoting the $L^2(\Omega)$ inner product as $(\cdot,\cdot)$ the weak formulation of the semi-discrete IFEM problem based on \eqref{eq:problem2} is\\
\noindent Find $u_h:J \longrightarrow S_{h,0}^1(\Omega)$ such that
\begin{equation}
\label{eq:problem3}
(c\pdd{t}{u_h}(t), v_h) +\mathcal{L}(u_h,v_h) = 0, \ \forall v_h \in S_{h,0}^1(\Omega).
\end{equation}
Let $u(t) \in H^{2,0}(\Omega)\cap H_0^{1,0}(\Omega)$ for $t \in J\setminus \{0\}$ be the solution to \eqref{eq:problem2}. The elliptic projection $P_hu$
is defined to be the solution to the auxilliary problem:

\noindent Find $P_hu:J\longrightarrow S_{h,0}^1(\Omega)$ such that $\forall v_h \in S_{h,0}^1(\Omega), \forall \ t \in J$,
\begin{equation}
\label{eq:aux}
\mathcal{L}(u(t)-P_hu(t),v_h) = 0; \implies \ \int_\Omega \kappa\pdd{x}{u(t)}\pdd{x}{v_h}dx = \int_\Omega\pdd{x}{P_hu(t)}\pdd{x}{v_h}dx;
\end{equation}
In addition we have $\forall v_h \in S_{h,0}^1(\Omega), \forall \ t \in J$,
\begin{equation}
\label{eq:aux2}
\mathcal{L}(\pdd{t}{u}(t)-\pdd{t}{(P_hu)}(t),v_h) = 0; \implies \ \int_\Omega \kappa\pdd{x}{}\left(\pdd{t}{u(t)}\right)\pdd{x}{v_h}dx = \int_\Omega\pdd{x}{}\left(\pdd{t}{P_hu(t)}\right)\pdd{x}{v_h}dx;
\end{equation}
Thus, $\pdd{t}{P_hu}(t)$ is the elliptic projection of $\pdd{t}{u}(t), \forall t \in J$. We have the following result

\begin{thm}
Let $\forall \ t \in J$, $u(t) \in H^{2,0}(\Omega)\cap H_0^{1,0}(\Omega)$, and $u_h(t) \in S_{h,0}^1(\Omega)$, be the solutions to problems \eqref{eq:problem2}, and \eqref{eq:problem3} respectively. Then $\exists$ a positive constant $C$ such that
\begin{equation}
\label{eq:result1}
||u-u_h||_{L^\infty(L^2(\Omega))} \leq ||(u-u_h)(0)||_{L^2(\Omega)}+C\rho h^2\left[|u|_{L^\infty(H^{2,0}(\Omega))}+\alpha_1|\pdd{t}{u}|_{L^1(H^{2,0}(\Omega))}\right],
\end{equation}
where
\begin{equation}
\label{eq:alpha}
\alpha_1 = \ds\frac{\mathrm{max}\{c^+,c^-\}}{\mathrm{min}\{c^+,c^-\}}
\end{equation}
and $\rho$ is defined in Theorem \ref{thm:3}
\end{thm}
\begin{proof}
The proof is quite standard and we just give the salient details here. We refer the reader to \cite{thomee2006galerkin} for details on similar proofs. We split the error into two parts
\begin{equation}
u(t)-u_h(t) = (u-P_hu)(t)+(P_hu-u_h)(t) = \eta(t)+\xi(t).
\end{equation}
From Theorem \ref{thm:3} we have $\forall \ t \in J$
\begin{align}
\label{eq:boundsoneta}
||\eta(t)||_{L^2(\Omega)} &\leq C\rho h^2|u(t)|_{H^{2,0}(\Omega)},\\
||\pdd{t}{\eta}(t)||_{L^2(\Omega)} &\leq C\rho h^2|\pdd{t}{u}(t)|_{H^{2,0}(\Omega)}.
\end{align}
To obtain bounds on $\xi$ we insert $\xi$ into the variational formulation \eqref{eq:problem3}. We have $\forall v_h \in S_{h,0}^1(\Omega)$,
\begin{equation}
\label{eq:1}
(c\pdd{t}{\xi},v_h)+\mathcal{L}(\xi,v_h)= (c\pdd{t}{(P_hu)},v_h)-(c\pdd{t}{(u_h)},v_h)+\mathcal{L}(P_hu,v_h)-\mathcal{L}(u_h,v_h).
\end{equation}
Using $\mathcal{L}(P_hu,v_h) = \mathcal{L}(u,v_h)$ and the identity $(c\pdd{t}{u_h},v_h)+\mathcal{L}(u_h,v_h) = 0$, $\forall v_h \in S_{h,0}^1(\Omega)$, we have $\forall v_h \in S_{h,0}^1(\Omega)$
\begin{equation}
\begin{split}
\label{eq:2}
(c\pdd{t}{\xi},v_h)+\mathcal{L}(\xi,v_h)&= (c\pdd{t}{(P_hu)},v_h)+\mathcal{L}(u,v_h)\\
& = (c\pdd{t}{(P_hu)},v_h)-(c\pdd{t}{(u)},v_h)\\
& = -(c\pdd{t}{\eta},v_h).
\end{split}
\end{equation}
Allowing $v_h = \xi$ and using the Cauchy-Schwarz inequality we have
\begin{equation}
\ds\frac{1}{2} \od{t}{||\sqrt{c}\xi(t)||_{L^2(\Omega)}^2}+\mathcal{L}(\xi(t),\xi(t)) \leq ||c\pdd{t}{\eta}(t)||_{L^2(\Omega)}||\xi(t)||_{L^2(\Omega)}.
\end{equation}
Since $\mathcal{L}(\xi(t),\xi(t))$ is nonnegative, by dropping this term and dividing by $||\xi(t)||_{L^2(\Omega)} \neq 0$ we obtain the stability result
\begin{equation}
\od{t}{||\xi(t)||_{L^2(\Omega)}}\leq \alpha_1||\pdd{t}{\eta}(t)||_{L^2(\Omega)},
\end{equation}
Integrating from 0 to $\tau \leq T$ we have
\begin{equation}
||\xi(\tau)||_{L^2(\Omega)} \leq ||\xi(0)||_{L^2(\Omega)}+\alpha_1\int_0^\tau||\pdd{t}{\eta}(s)||_{L^2(\Omega)}ds.
\end{equation}
From Theorem \ref{thm:3} we have
\begin{equation}
\begin{split}
||\xi(0)||_{L^2(\Omega)} = ||(P_hu-u_h)(0)||_{L^2(\Omega)} &\leq ||(u-u_h)(0)||_{L^2(\Omega)}+||(u-P_hu)(0)||_{L^2(\Omega)}\\
&\leq ||(u-u_h)(0)||_{L^2(\Omega)}+C\rho h^2|u(0)|_{H^{2,0}(\Omega)}.
\end{split}
\end{equation}
Using Theorem \ref{thm:3} for $\pdd{t}{u}$ and $\pdd{t}{(P_hu)}$ to get
\begin{equation}
||\xi||_{L^{\infty}(L^2(\Omega))} \leq ||(u-u_h)(0)||_{L^2(\Omega)} +C\rho h^2\left(|u(0)|_{H^{2,0}(\Omega)}+\alpha_1\int_0^\tau|\pdd{t}{u}(s)|_{H^{2,0}(\Omega)}ds\right).
\end{equation}
Combining the estimates on $\xi$ and $\eta$ together gives us the result \eqref{eq:result1}.
\end{proof}

Estimates in space in the semi-norm $|\cdot|_{H^{1,0}(\Omega)}$ can also be derived by choosing $v_h = \xi_t$ in \eqref{eq:2}. See \cite{attanayake2011convergence} for details.
\subsection{Fully Discrete Schemes: Error Estimates}
In this section we develop error estimates for the fully discrete numerical scheme obtained by applying a backward Euler discretization or a Crank-Nicolson update in time. Given a time step $\Delta t > 0$ we define discrete time levels $t_k = k\Delta t$ for $k = 0,1,2,\ldots, M$ with $t_M =M\Delta t = T$.  The fully discrete solution at $t_k$ is denoted as $u_h^k$.

\subsubsection{Discretization in Time with $\theta$ Schemes}
We consider a one parameter family of finite difference discretizations in time called $\theta$ schemes.
The fully discrete variational problem using a $\theta$ scheme in time is:\\
Find $u_h^k \in  S_{h,0}^1(\Omega), k = 1,2,\ldots,M$ such that
\begin{equation}
\label{eq:problem3-timediscrete}
\begin{split}
&\left(c(x)\ds\frac{u_h^k-u_h^{k-1}}{\Delta t}, v_h\right) +\mathcal{L}(u_h^{k-\theta},v_h) = 0, \ \forall v_h \in S_{h,0}^1(\Omega), \forall \ k = 1,2,\ldots, M,\\
&(u_h^0-u_0,v_h) = 0, \ \forall v_h \in S_{h,0}^1(\Omega),
\end{split}
\end{equation}
where $0\leq \theta\leq 1$ and
\begin{equation}
u_h^{k-\theta} = \theta u_h^k+(1-\theta)u_h^{k-1}.
\end{equation}
Thus, if $\theta = 0$ we obtain the forward Euler method in time, if $\theta = 1$ we obtain the backward Euler method and for $\theta = \frac{1}{2}$ we obtain the Crank-Nicolson scheme. Here we consider $\theta = \frac{1}{2}$ and develop the error estimates for the IFEM method with a Crank Nicolson time discretization. For other values of $\theta \in [0,1]$ the analysis is analogous \cite{Suli,thomee2006galerkin}.

\subsubsection{Crank-Nicolson Discretization in Time}

The fully discrete variational problem using a Crank-Nicolson time discretization is:\\
Find $u_h^k \in  S_{h,0}^1(\Omega)$ such that
\begin{equation}
\label{eq:problem3-CN}
\left(c(x)\ds\frac{u_h^k-u_h^{k-1}}{\Delta t}, v_h\right) +\mathcal{L}(\ds\frac{u_h^k+u_h^{k-1}}{2},v_h) = 0, \ \forall v_h \in S_{h,0}^1(\Omega), \forall \ k = 1,2,\ldots, M,
\end{equation}

The fully discrete variational problem \eqref{eq:problem3-CN} satisfies the following error estimate.
\begin{thm}
\label{thm:errorCN}
Let $u(t_k) \in H^{2,0}(\Omega)\cap H_0^1(\Omega)$ and $u_h^k \in S_{h,0}^1(\Omega); k = 1,2,\ldots, M$ be the solutions to problems \eqref{eq:problem2} and \eqref{eq:problem3-CN}. Then $\exists$ a positive constant $C$ such that
\begin{equation}
\label{eq:result4CN}
\begin{split}
{\underset{0 \le k \le M}{\mathrm{max}}}||u(t_k)-u_h^k||_{L^2(\Omega)} &\leq C\rho h^2\left(\alpha_3^T{\underset{0 \le k \le M}{\mathrm{max}}}|u(t)|_{H^{2,0}(\Omega)}+\alpha_4\int_0^{T}|\pdd{t}{u}(s)|_{H^{2,0}(\Omega)}ds\right)\\
&+\alpha_4\frac{\Delta
t^2}{2}\left(\frac{1}{4}\int_0^{T}||\frac{\partial^3u}{\partial t^3}(s)||_{L^2(\Omega)}ds+\alpha_5 \int_0^{T}|\frac{\partial^2u}{\partial t^2}(s)|_{H^{2,0}(\Omega)}ds\right)\\
&+\alpha_3^T||u(0)-u_h^0||_{L^2(\Omega)}.
\end{split}
\end{equation}
where $\alpha_3$, $\alpha_4$ and $\alpha_5$  are defined as
\begin{align}
\label{eq:alpha3}
\alpha_3 &= \ds\frac{\mathrm{max}\{\sqrt{c^+},\sqrt{c^-}\}}{\mathrm{min}\{\sqrt{c^+},\sqrt{c^-}\}},\\
\label{eq:alpha4}
\alpha_4 &= {\underset{0 \le \ell \le k-1}{\mathrm{max}}}\{ \alpha_3^\ell\},\\
\label{eq:alpha5}
\alpha_5 &= \ds\frac{\mathrm{max}\{\sqrt{\kappa^+},\sqrt{\kappa^-}\}}{\mathrm{min}\{\sqrt{c^+},\sqrt{c^-}\}},
\end{align}
\end{thm}
\begin{proof}
We split the error into two parts
\begin{equation}
u(t_k)-u_h^k = (u(t_k)-P_hu(t_k))+(P_hu(t_k)-u_h^k) = \eta_k+\xi_k, \forall \ k=1,2,\ldots, M.
\end{equation}
We insert $\xi_k$ into the fully discrete variational formulation \eqref{eq:problem3-CN}, add and subtract the term $c\frac{u(t_k)-u(t_{k-1})}{\Delta t}$, use $\mathcal{L}(P_hu(t_k),v_h) = \mathcal{L}(u(t_k),v_h)$ and the identity \eqref{eq:problem3-CN} to get
\begin{equation}
\label{eq:discrete2CN}
\begin{split}
\left(c\ds\frac{\xi_k-\xi_{k-1}}{\Delta t},v_h\right)+\mathcal{L}(\ds\frac{\xi_k+\xi_{k-1}}{2},v_h)= -(c(w^1_k+w^2_k+w^3_k),v_h), \forall v_h \in S_{h,0}^1(\Omega),
\end{split}
\end{equation}
where
\begin{align}
w_k^1 &= \left[\left(\pdd{t}{u}(t_{k-\frac{1}{2}}) - \ds\frac{u(t_k)-u(t_{k-1})}{\Delta t}\right)\right],\\
w_k^2 &= \left[\left(\ds\frac{u(t_k)-u(t_{k-1})}{\Delta t}-\ds\frac{P_hu(t_k)-P_hu(t_{k-1})}{\Delta t}\right)\right],
\end{align}
and $w_k^3$ is defined through the auxilliary problem
\begin{equation}
-(cw_k^3,v_h) = \mathcal{L}(\ds\frac{u(t_k)+u(t_{k-1})}{2}-u(t_{k-\frac{1}{2}}),v_h), \forall \ v_h \in S_{h,0}^1(\Omega).
\end{equation}

Choose $v_h = \frac{(\xi_k+\xi_{k-1})}{2}$ in \eqref{eq:discrete2CN} to get
\begin{equation}
\begin{split}
||\sqrt{c}\xi_k||_{L^2(\Omega)}^2 -||\sqrt{c}\xi_{k-1}||_{L^2(\Omega)}^2 &+\Delta t |\sqrt{\kappa}(\xi_k+\xi_{k-1})|_{H^{1,0}(\Omega)}^2 \\
\leq \Delta t \left(||\sqrt{c}\xi_k||_{L^2(\Omega)}\right.&+\left.||\sqrt{c}\xi_{k-1})||_{L^2(\Omega)}\right)||\sqrt{c}(w^1_k+w^2_k+w^3_k)||_{L^2(\Omega)}.
\end{split}
\end{equation}
By dropping the nonnegative term $\Delta t |\sqrt{\kappa}(\xi_k+\xi_{k-1})|_{H^{1,0}(\Omega)}^2$ and dividing by $(||\sqrt{c}\xi_k||_{L^2(\Omega)}+||\sqrt{c}\xi_{k-1}||_{L^2(\Omega)})$ we have
\begin{equation}
\label{eq:discrete1CN}
||\xi_k||_{L^2(\Omega)} \leq \ds\frac{\mathrm{max}\{\sqrt{c^+},\sqrt{c^-}\}}{\mathrm{min}\{\sqrt{c^+},\sqrt{c^-}\}}\left[||\xi_{k-1}||_{L^2(\Omega)}+\Delta t ||w^1_k+w^2_k+w^3_k||_{L^2(\Omega)}\right].
\end{equation}
Applying \eqref{eq:discrete1CN} recursively, using the definition \eqref{eq:alpha3} of $\alpha_3$ and using the definition of $\alpha_4$ from \eqref{eq:alpha4} gives us the discrete stability estimate
\begin{equation}
\label{eq:discretestabilityCN}
||\xi_k||_{L^2(\Omega)} \leq \alpha_3^k||\xi_0||_{L^2(\Omega)}+\Delta t \ \alpha_4\left[\sum_{\ell = 1}^{k}\left(||w^1_{\ell}||_{L^2(\Omega)}+||w_\ell^2||_{L^2(\Omega)}+||w_\ell^3||_{L^2(\Omega)}\right)\right].
\end{equation}
From Taylor's formula we can show that
\begin{equation}
\Delta t\sum_{\ell = 1}^{k}||w_\ell^1||_{L^2(\Omega)} \leq \frac{\Delta t^2}{8} \int_0^{t_k}||\frac{\partial^3u}{\partial t^3}(s)||_{L^2(\Omega)}ds.
\end{equation}

Next, we have
\begin{equation}
w_\ell^2 = \left(\ds\frac{u(t_\ell)-P_hu(t_\ell)}{\Delta t}\right)-\left(\ds\frac{u(t_{\ell-1})-P_hu(t_{\ell-1})}{\Delta t}\right)
=\ds\frac{1}{\Delta t}\int_{t_{\ell-1}}^{t_\ell}\pdd{t}{(u-P_hu)}(s)ds.
\end{equation}
From Theorem \ref{thm:3} we have
\begin{equation}
\Delta t \sum_{\ell = 1}^{k}||w_\ell^2||_{L^2(\Omega)} \leq \sum_{\ell = 1}^k\int_{t_{\ell-1}}^{t_\ell} C\rho h^2|\pdd{t}{u}(s)|_{H^{2,0}(\Omega)}ds\leq C\rho h^2\int_0^{t_k}|\pdd{t}{u}(s)|_{H^{2,0}(\Omega)}ds.
\end{equation}

Finally we have
\begin{equation}
||w_k^3||_{L^2(\Omega)} \leq \alpha_5||\frac{\partial^2}{\partial x^2}\left(\ds\frac{u(t_k)+u(t_{k-1})}{2}-u(t_{k-\frac{1}{2}})\right)||_{L^2(\Omega)}
\end{equation}
and thus
\begin{equation}
\Delta t \sum_{\ell = 1}^{k}||w_\ell^3||_{L^2(\Omega)} \leq \alpha_5\frac{\Delta t^2}{2} \int_0^{t_k}|\frac{\partial^2u}{\partial t^2}(s)|_{H^{2,0}(\Omega)}ds.
\end{equation}

Thus, $\forall \ k = 1,2,\ldots, M$ we have
\begin{equation}
\label{eq:boundsonxiCN}
\begin{split}
||\xi_k||_{L^2(\Omega)} &\leq
\alpha_3^k||u(0)-u_h^0||_{L^2(\Omega)}+C\rho h^2\left(\alpha_3^k|u(0)|_{H^{2,0}(\Omega)}+\alpha_4\int_0^{t_k}|\pdd{t}{u}(s)|_{H^{2,0}(\Omega)}ds\right)\\
&+\alpha_4\frac{\Delta
t^2}{2}\left(\frac{1}{4}\int_0^{t_k}||\frac{\partial^3u}{\partial t^3}(s)||_{L^2(\Omega)}ds+\alpha_5 \int_0^{t_k}|\frac{\partial^2u}{\partial t^2}(s)|_{H^{2,0}(\Omega)}ds\right).
\end{split}
\end{equation}
Combining the bounds \eqref{eq:boundsonxiCN} on $\xi_k$ with the bounds on $\eta_k$ from \eqref{eq:boundsoneta} with $t=t_k$ we finally obtain the result \eqref{eq:result4CN} of the theorem.
\end{proof}

\section{Numerical Examples of the Deterministic Methods}
\label{sec:numd}

Consider the initial profile given by\gibson{}
\begin{equation}\label{eq:example}
u_0(x)=\begin{cases}
(1-x^2)^5&\mbox{if }|x|<1\\
0&\mbox{else}.
\end{cases}
\end{equation}
In this section simulations are provided of the solution to
\eqref{eq:diff} with \eqref{eq:example} for values of
$D^+=\{10,100\}$ while holding $D^-=1$. We consider scenarios with
$\lambda=\{\lambda^*, \lambda^\#\}$. Using the IFEM-CN method, the
expected value solution formula from \cite{Appuhamilage_AAP}, and the
SDE-Euler-Maruyama method discussed in Section \ref{sec:Stochastic},
all computed at $t=0.2$, are shown in Figure \ref{fig:solution}. As
the simulations using IFEM-BE are indistinguishable, only the
IFEM-CN are displayed in these plots.

In each of the above cases, the error is computed between the numerical approximation and
the expected value solution formula on the interval $[-5, 5]$.  (Note that in the case of IFEM, the solution was computed on a larger interval in order to avoid contamination from boundary effects.)  The
two-norm of the error (in space, infinity-norm in time) is plotted versus the spatial step $h$ on a log-log plot
demonstrating second order (spatial) accuracy.  In the case of Backward Euler (Figure \ref{fig:IFEMBEerror}), the time step was chosen to be $O(h^2)$, whereas for Crank-Nicolson (Figure \ref{fig:IFEMCNerror}) $\Delta t=O(h)$.

\begin{figure}[!t]
  \begin{center}\hspace*{-12pt}
    \begin{tabular}{c}
      \mbox{\includegraphics[width = 0.5\textwidth, angle=0]{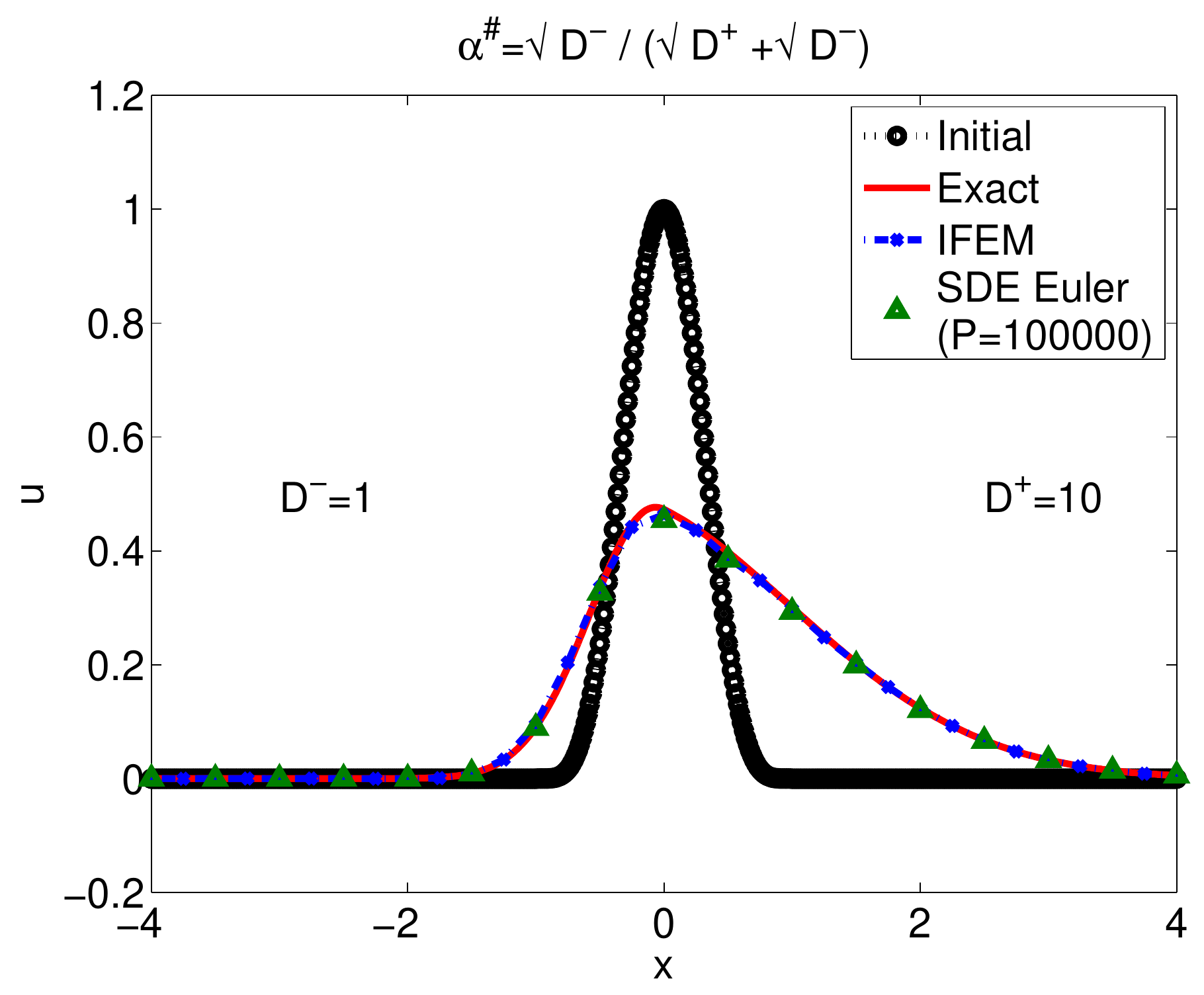}} \
      \mbox{\includegraphics[width = 0.5\textwidth, angle=0]{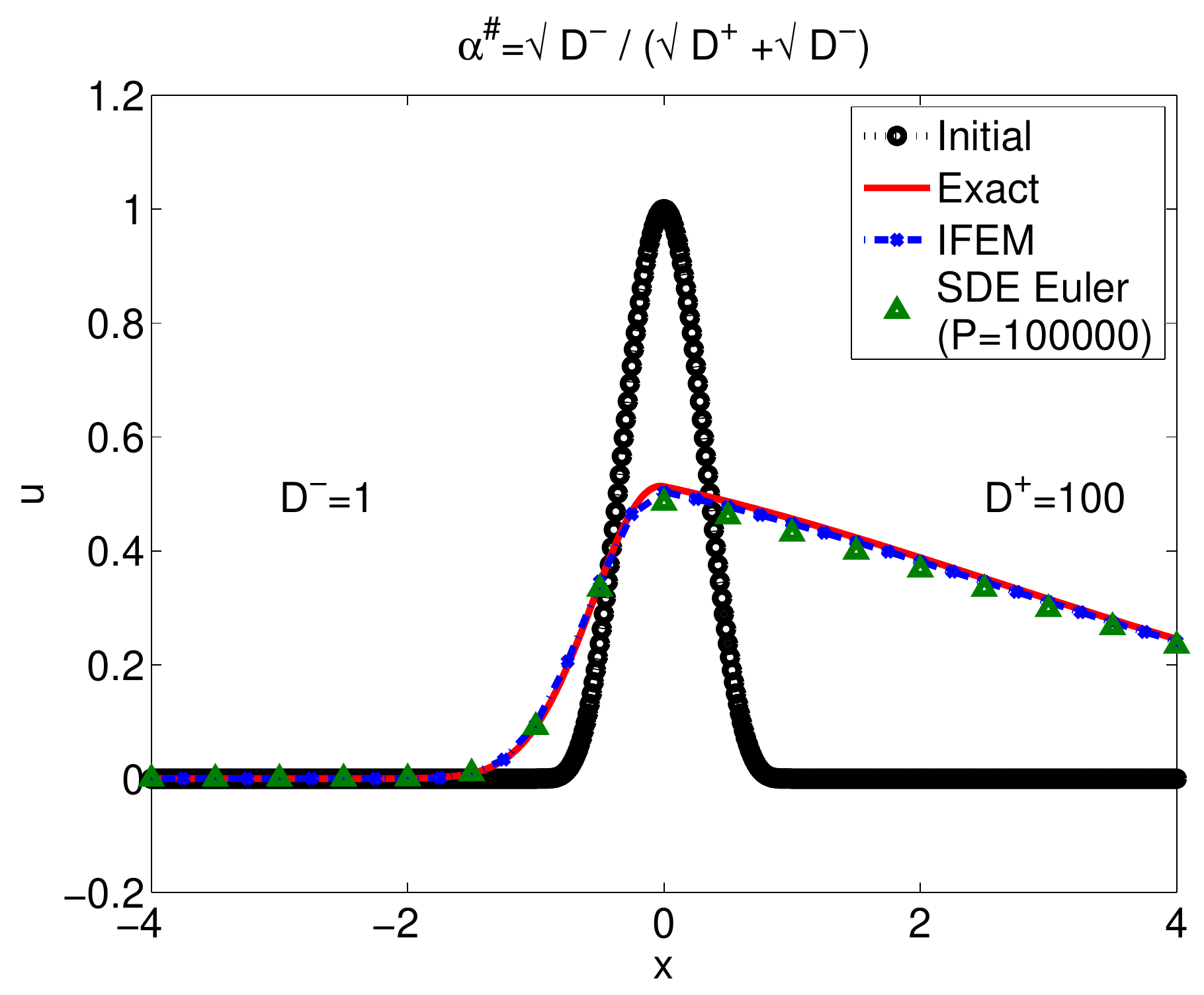}} \\
      \mbox{\includegraphics[width = 0.5\textwidth, angle=0]{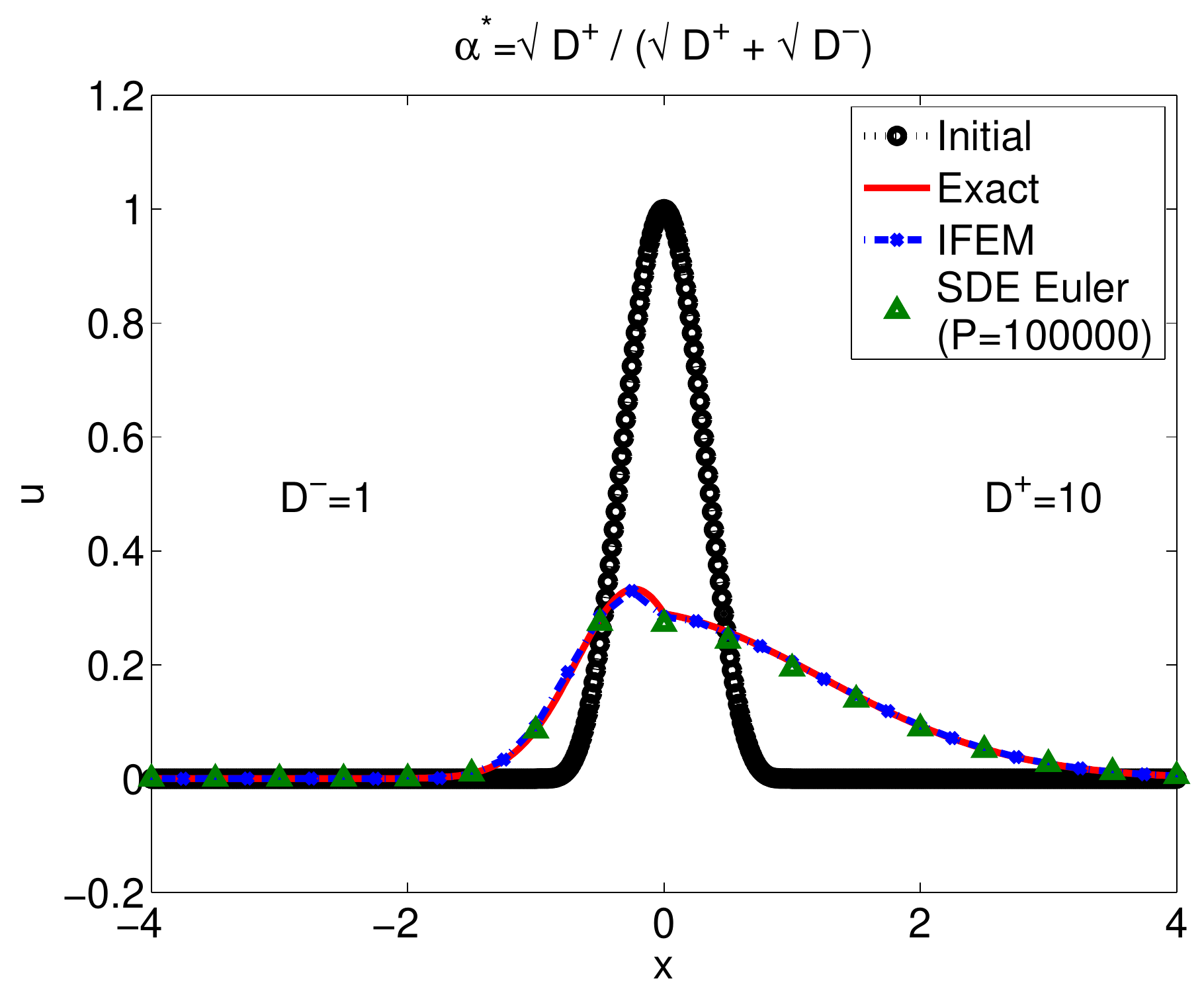}} \
      \mbox{\includegraphics[width = 0.5\textwidth, angle=0]{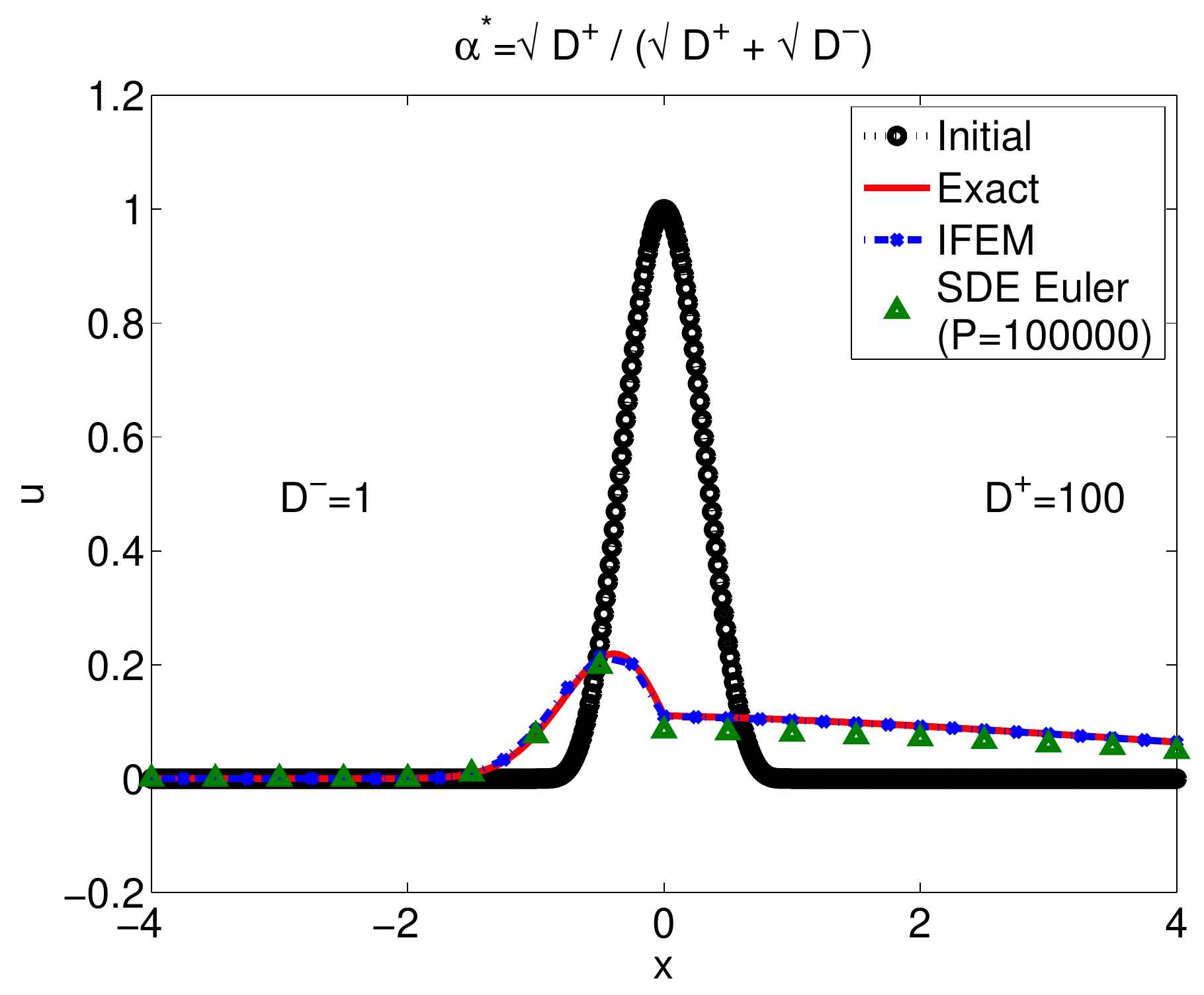}} \\
      \end{tabular}
  \end{center}\caption{Initial and computed solution of \eqref{eq:example} at $t=0.2$ using IFEM-CN method, the SDE-Euler method and the expected value solution formula from \cite{Appuhamilage_AAP}, for various combinations of $D^+$ and $\lambda$ values. }\label{fig:solution}
\end{figure}

\begin{figure}[!t]
  \begin{center}
    \begin{tabular}{c}
      \mbox{\includegraphics[width = 0.5\textwidth, angle=0]{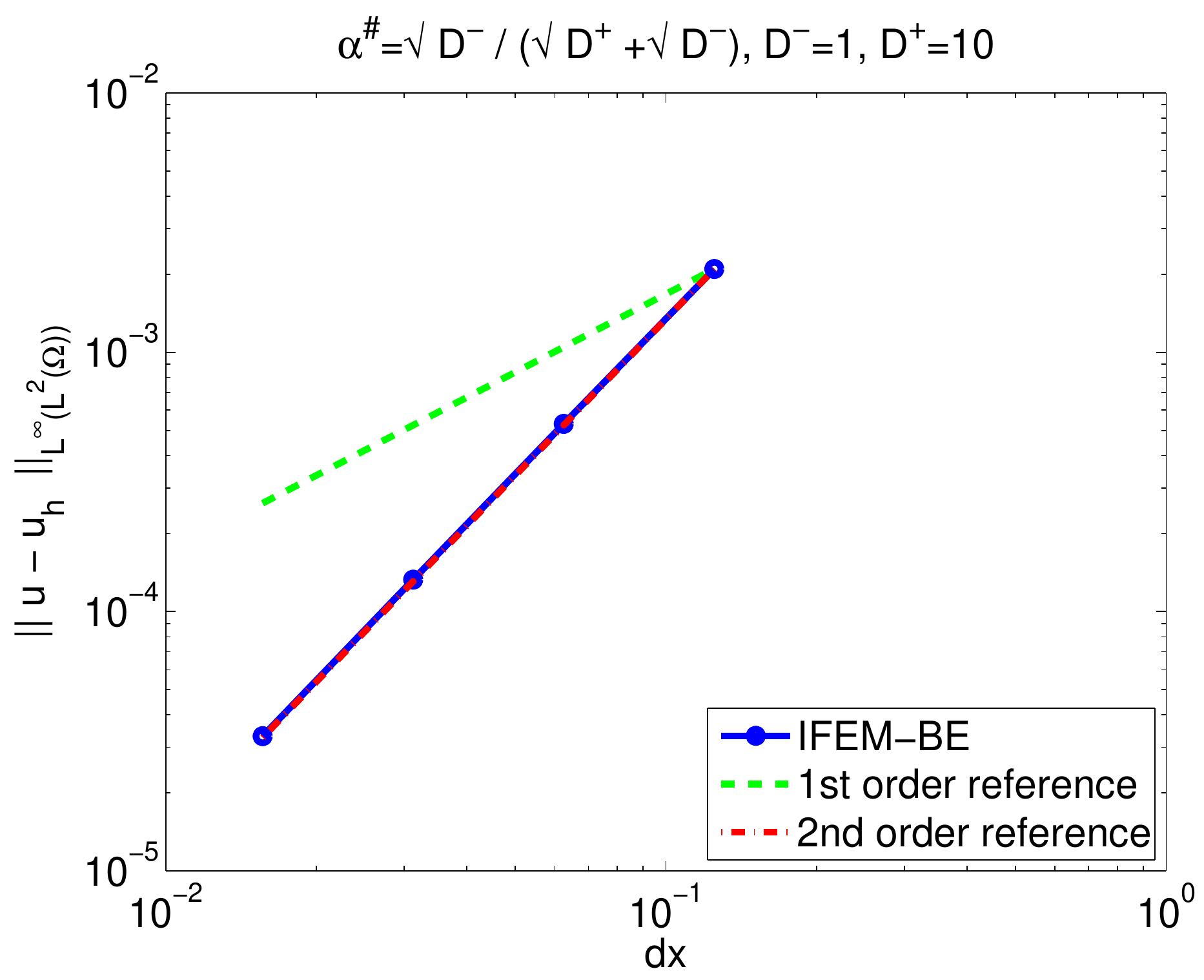}} \
      \mbox{\includegraphics[width = 0.5\textwidth, angle=0]{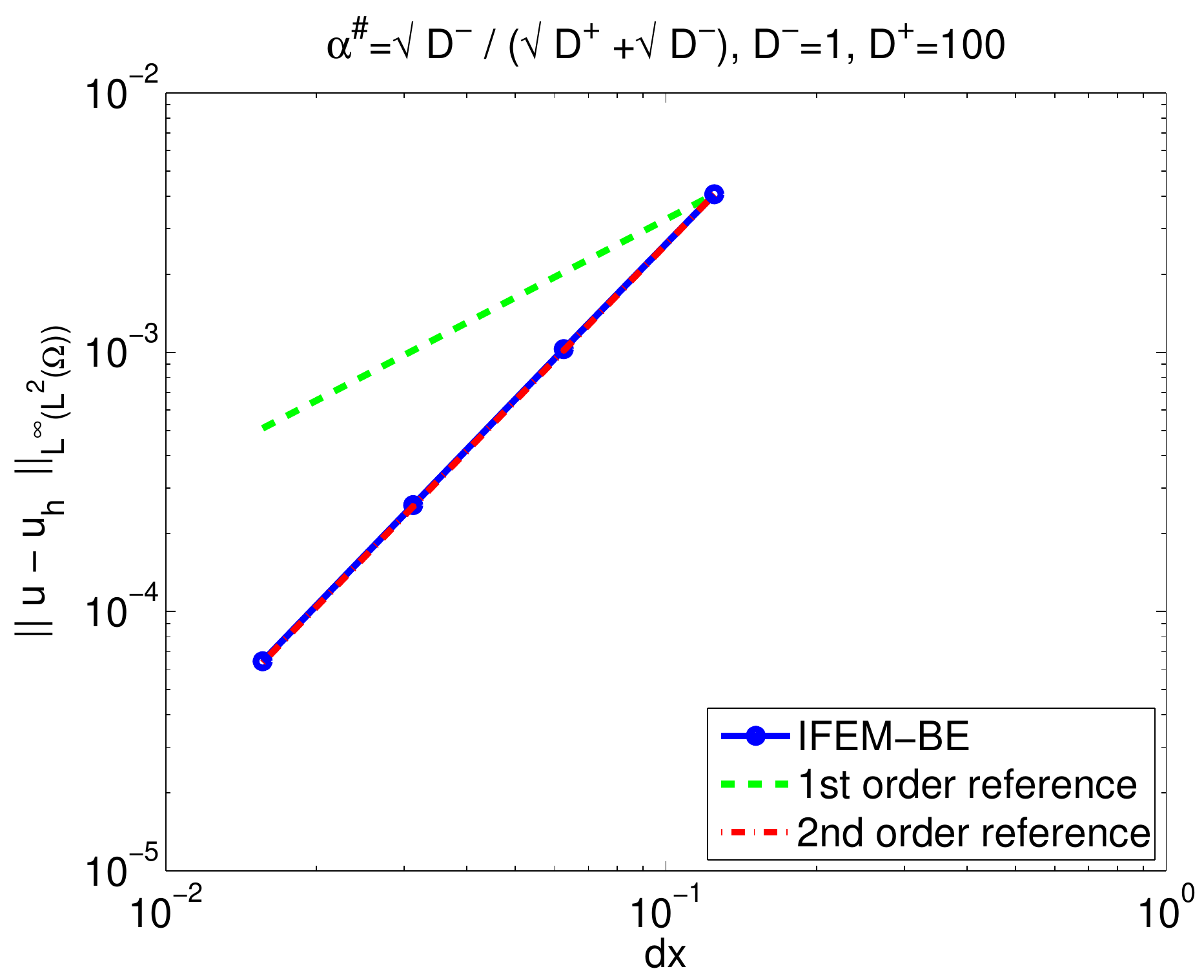}} \\
      \mbox{\includegraphics[width = 0.5\textwidth, angle=0]{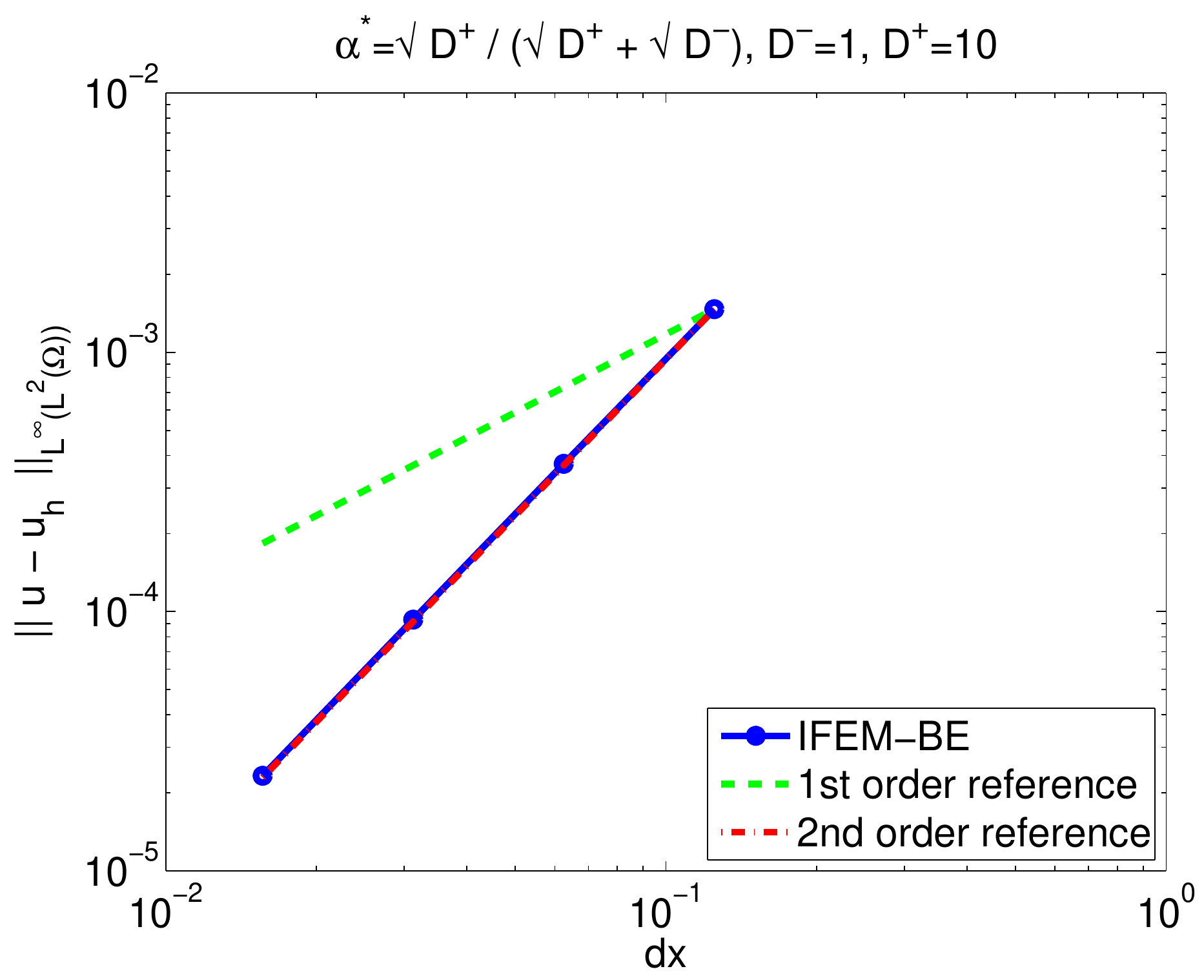}} \
      \mbox{\includegraphics[width = 0.5\textwidth, angle=0]{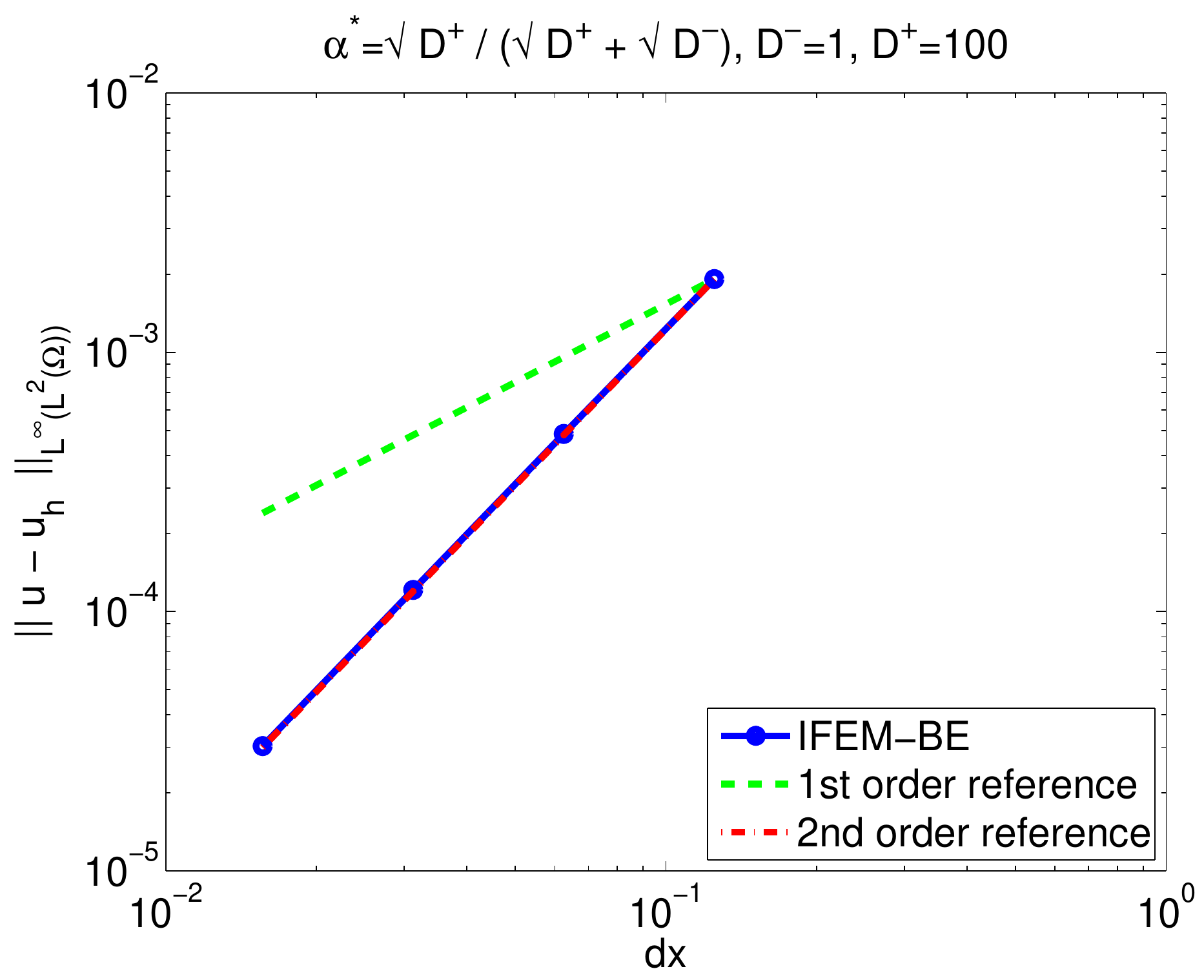}} \\
      \end{tabular}
  \end{center}\caption{Error for the IFEM-BE method demonstrating second order convergence for the above scenarios.}\label{fig:IFEMBEerror}
\end{figure}

\begin{figure}[!t]
  \begin{center}
    \begin{tabular}{c}
      \mbox{\includegraphics[width = 0.5\textwidth,angle=0]{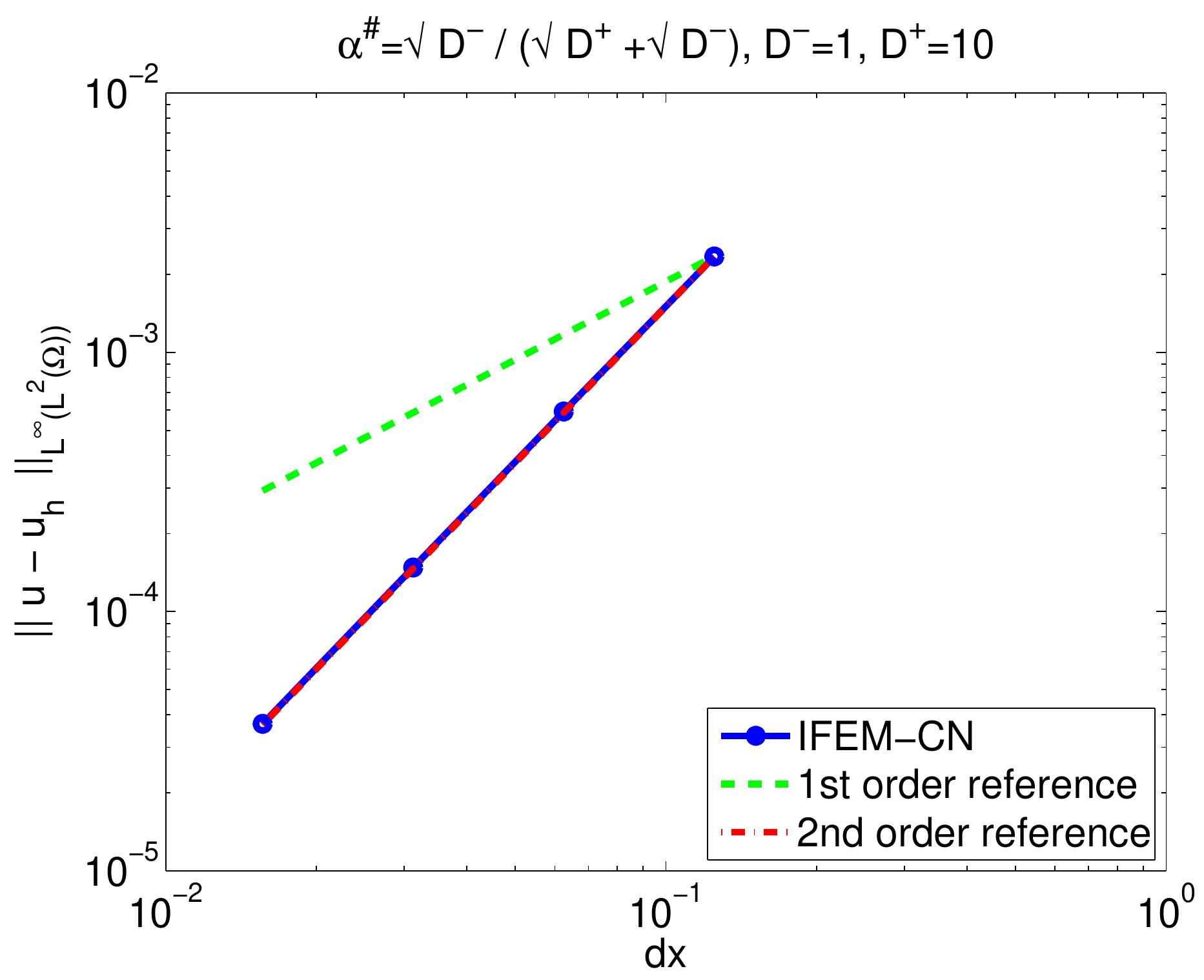}} \
      \mbox{\includegraphics[width = 0.5\textwidth,angle=0]{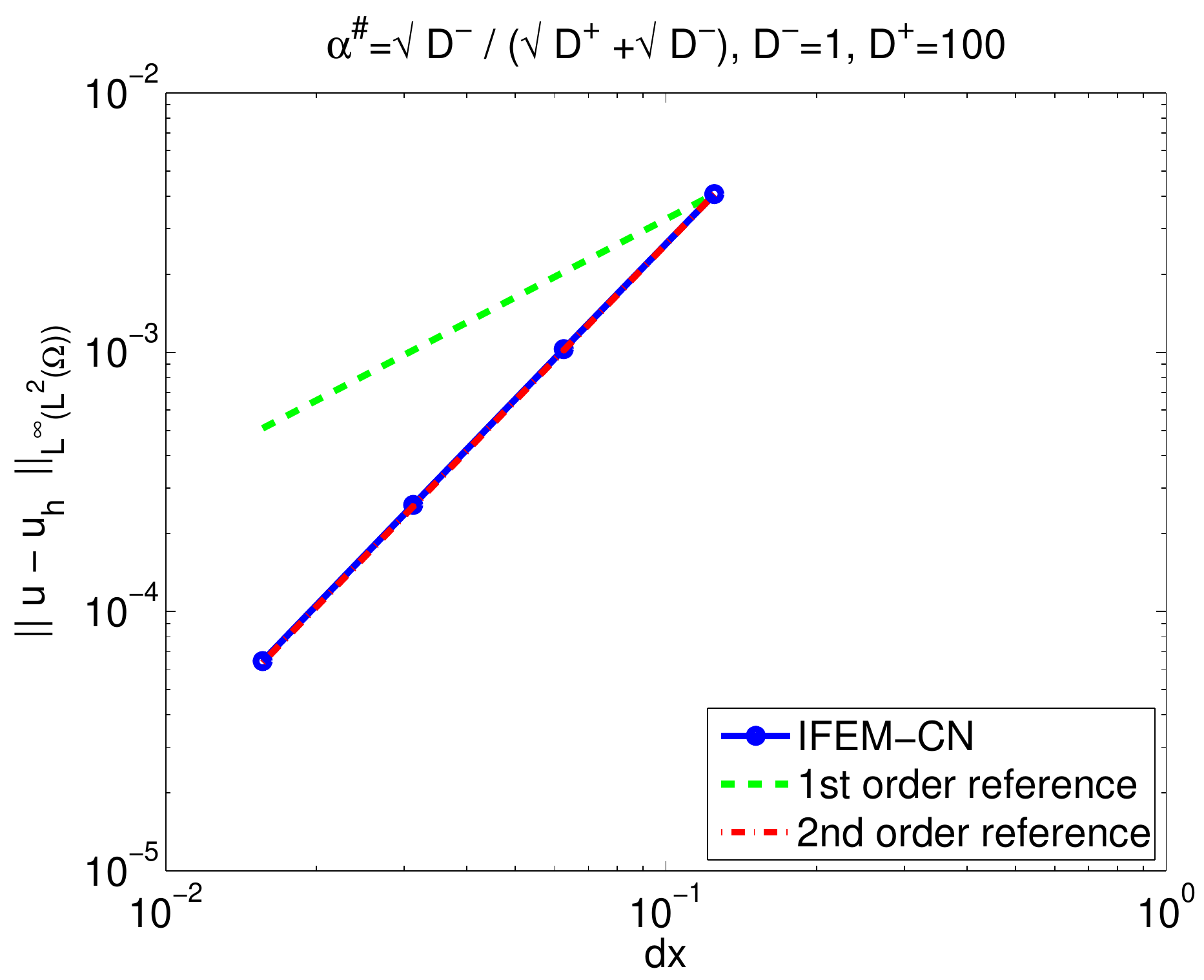}} \\
      \mbox{\includegraphics[width = 0.5\textwidth,angle=0]{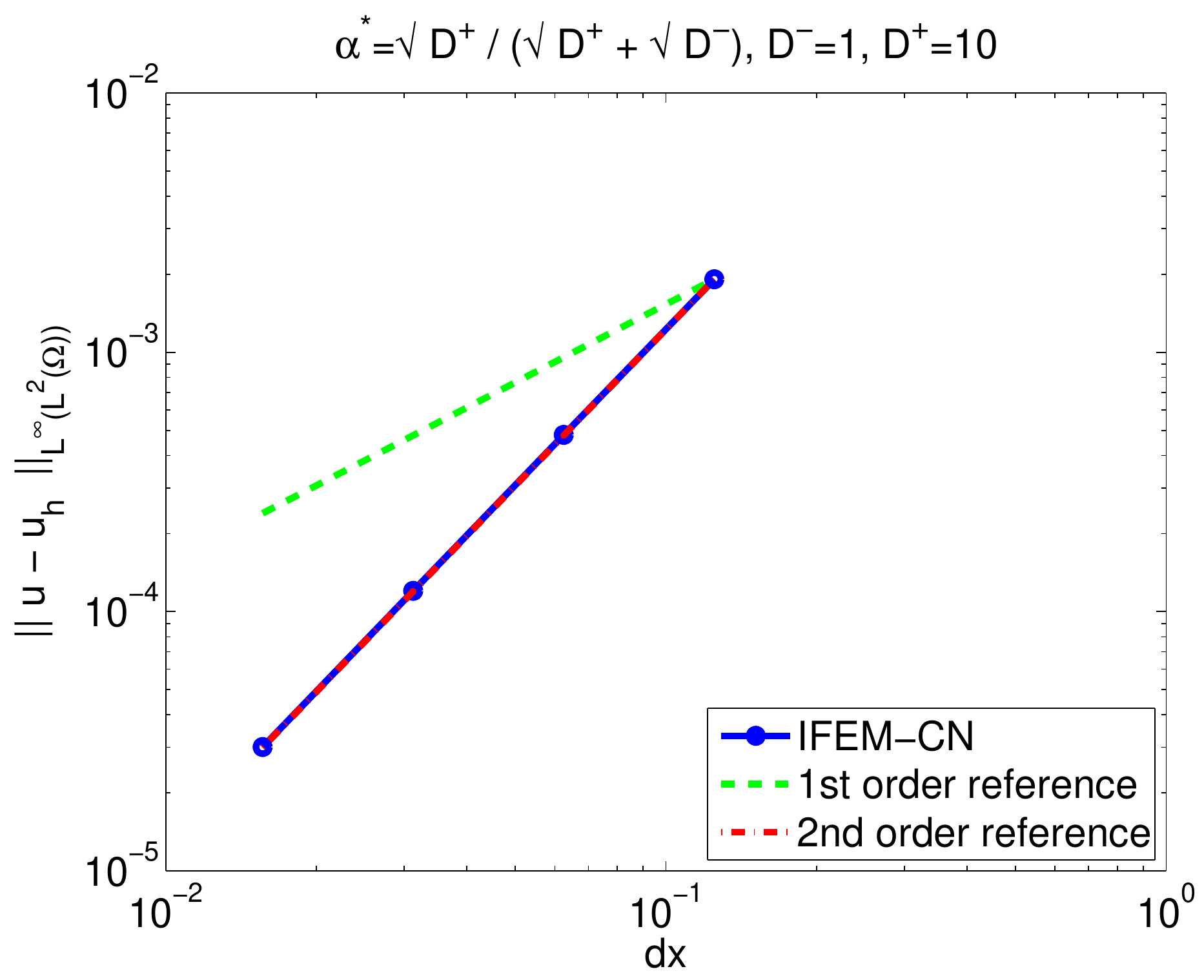}} \
      \mbox{\includegraphics[width = 0.5\textwidth,angle=0]{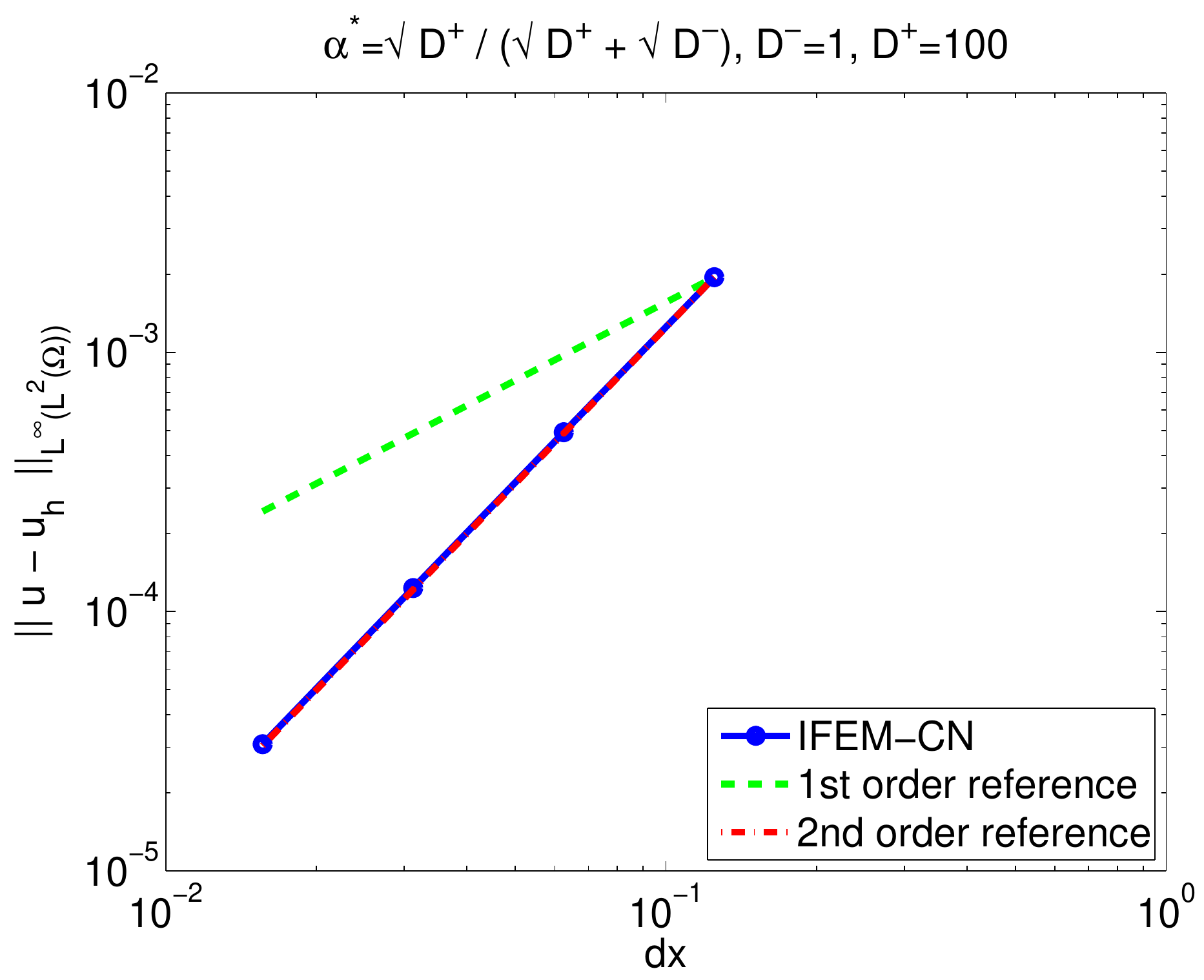}} \\
    \end{tabular}
  \end{center}\caption{Error for the IFEM-CN method demonstrating second order convergence for the above scenarios.}\label{fig:IFEMCNerror}
\end{figure}

\section{Numerical Methods for Stochastic Diffusion in the Presence of an Interface}
\label{sec:Stochastic}

In this section we will consider a numerical solution to system (2.1) using
a Monte-Carlo method.  The discontinuities in the coefficient of the equation,
as well as the generality of the interface condition considered in this paper present challenges in two
different aspects of the theory.  On the one hand, the discontinuity in the diffusion coefficient naturally requires to consider SDE's that include a local time term (see section 2.1 for details.)  As noted in \cite{MartinezT12}, a transformation
of the stochastic process can be defined so that this local time term is eliminated.  On the other hand, the generality of the interface condition renders inadequate the approach of \cite{MartinezT12} since they benefited from the self adjoint
property of the problem under their consideration.  Instead, in the problem
consider in this paper, a careful quantification of the effect of the interface
condition is needed.

The organization of this section is as follows.  In section 5.1 we review basic
aspects of Skew Brownian motion, review details of the stochastic
representation of solutions of (2.1) obtained in \cite{appuhamillage2012interfacial},
and obtain basic estimates on the corresponding transition probability densities.
In section 5.2 we follow a similar approach as the one developed in \cite{MartinezT12} to
eliminate the local time term in the SDE associated to solutions of (2.1),
an introduce an Euler-Maruyama method to approximate solutions of the
resulting SDE.  The main theorems establishing the rate of convergence of the
approximation are stated in this section, with proofs given in section 5.3.

\subsection{Stochastic Representation of the Solution to \eqref{eq:diff}.}
Let us first record a definition of skew Brownian motion
$B^{(\alpha)}(t)$, $0<\alpha<1$, originally  introduced by It\^o and
McKean. Let $|B(t)|$ denote the reflecting Brownian motion starting
at $0$, and enumerate the excursion intervals away from $0$ by $J_1,
J_2,...$. Let $A_1,A_2,...$ be an i.i.d. sequence of Bernoulli
$\pm1$ random variables, independent of $B(t)$, with
$P(A_n=1)=\alpha$. Then $B^{(\alpha)}(t)$ is defined by changing the
signs of the excursion over the intervals $J_n$ whenever to
$A_n=-1$, for $n=1,2,...$. That is
\begin{equation}\label{skewBM defn}
B^{(\alpha)}(t)=\sum_{n=1}^\infty A_n\mathbf{1}_{J_n}(t)|B(t)|,\quad
t\ge0.
\end{equation}

 Denote
 $\sigma(x)=\sqrt{D^+}x\mathbf{1}_{[0,\infty)}(x)+\sqrt{D^-}x\mathbf{1}_{(-\infty,0)}(x)$
 and
 \begin{equation}\label{Y^alpha-defn}
Y^{(\alpha)}(t)=\sigma\big(B^{(\alpha)}(t)\big),\quad(t\ge0).
 \end{equation}
It follows from \cite[Theorem 3.1]{appuhamillage2012interfacial} that if $f\in
C^2(\mathbb{R}\setminus\{0\})\cap C(\mathbb{R})$ satisfying the
condition $\lambda f'(0^+)=(1-\lambda)f'(0^-)$ then, for
\begin{equation}\label{alpha}
\alpha=\alpha(\lambda)={\lambda\sqrt{D^-}\over
\lambda\sqrt{D^-}+(1-\lambda)\sqrt{D^+}}
\end{equation}
we have
\begin{equation}\label{martingale-prob}
f\big(Y^{(\alpha)}(t)\big)=f\big(Y^{(\alpha)}(0)\big)+\int_0^tf'_-\big(Y^{(\alpha)}(s)\big)\sqrt{D\big(Y^{(\alpha)}(s)\big)}dB(s)
+{1\over2}\int_0^tD\big(Y^{(\alpha)}(s)\big)f''\big(Y^{(\alpha)}(s)\big)ds.
\end{equation}
In addition, $Y^{(\alpha)}(t)$ satisfies the following stochastic
differential equation with a local time
\begin{equation}\label{Ito-eqn}
dY^{(\alpha)}(t)=\sqrt{D\big(Y^{(\alpha)}(t)\big)}dB(t)
+\Big({\sqrt{D^+}-\sqrt{D^-}\over2}+\sqrt{D^-}{2\alpha-1\over2\alpha}\Big)dl_t^{B^{(\alpha)},+}(0)
\end{equation}
where $D(x)$ is defined as in \eqref{D(x)} and the local time $l_t^{B^{(\alpha)},+}(0)$ is defined by
$$
l_t^{B^{(\alpha)},+}(0)=\lim_{\epsilon\downarrow0}{1\over\epsilon}\int_0^t\mathbf{1}_{[0,\epsilon)}(B^{(\alpha)}(s))d\langle
B^{(\alpha)}\rangle_s.
$$

For each $g\in C^2_b(\mathbb{R}\setminus\{0\})$ and $x\ne0$ we
denote the operator
\begin{equation}\label{operator L}
\tilde{ \mathcal{L}}g(x)={D(x)\over2}g''(x).
\end{equation}
In addition, denote
\begin{align}
\mathcal{W}^2&=\{g\in C^2_b(\mathbb{R}\setminus\{0\}): g^{(i)}\in L^1(\mathbb{R})\cap L^2(\mathbb{R}), i=1,2; \lambda g'(0^+)=(1-\lambda)g'(0^-)\},\label{W2 defn}\\
\mathcal{W}^4&=\{g\in C^4_b(\mathbb{R}\setminus\{0\}): g^{(i)}\in
L^1(\mathbb{R})\cap L^2(\mathbb{R}), i=1, ..., 4; \lambda
g'(0^+)=(1-\lambda)g'(0^-),\notag\\
&\hspace{7.87cm}\lambda
(\tilde{\mathcal{L}}g)'(0^+)=(1-\lambda)(\tilde{\mathcal{L}}g)'(0^-)\}.\label{W4
defn}
\end{align}

Now we are in a position to state the stochastic representation
theorem which can be found in \cite{appuhamillage2012interfacial} (see also
\cite{MartinezT12})
\begin{thm}[Corollary 3.2 from \cite{appuhamillage2012interfacial}]\label{thm:stoch-representation} Let
$0<\lambda<1$, $\alpha=\alpha(\lambda)$ as in \eqref{alpha} and $u_0\in \mathcal{W}^2$. Then the function
$u(t,x)=E^xu_0(Y^{(\alpha)}(t))$, where
$(t,x)\in[0,T]\times\mathbb{R}$, is the unique function in
$C^{1,2}_b([0,T]\times(\mathbb{R}\setminus\{0\}))\cap
C([0,T]\times\mathbb{R})$ which satisfies the equations
\eqref{eq:diff}.
\end{thm}

Next, we have some pointwise estimates for the derivatives of
$u(t,x)$. A similar result for the case of $\lambda={D^+\over
D^++D^-}$ (continuity of flux) was given in \cite{MartinezT12}.

\begin{thm}\label{thm:derv-est} (i) Let
$0<\lambda<1$, $\alpha=\alpha(\lambda)$ as in \eqref{alpha}.
Then the probability distribution of $Y^{(\alpha)}(t)$ under $P^x$
(i.e. $Y^{(\alpha)}(0)=x$) has a density $q^{(\alpha)}(t,x,y)$ which
satisfies
\begin{itemize}
\item There exists $C>0$ such that for all $x\in\mathbb{R}$, $t>0$ and for
Lebesgue a.s. $y\in\mathbb{R}\setminus\{0\}$,
\begin{equation}\label{density-est}
q^{(\alpha)}(t,x,y)\le{C\over\sqrt t}.
\end{equation}
\item There exists $C>0$ such that for all $x\in\mathbb{R}$, $t>0$ and $u_0\in L^1(\mathbb{R})$,
\begin{equation}\label{exptn-est}
|E^xu_0\big(Y^{(\alpha)}(t)\big)|\le{C\over\sqrt t}\|u_0\|_1.
\end{equation}
\end{itemize}
(ii) For all $j=0,1,2$ and $i=1, 2, 3, 4$ satisfying $2j+i\le 4$
there exists $c>0$ such that for all  $x\in\mathbb{R}$, $t>0$ and
$u_0\in \mathcal{W}^4$,
\begin{equation}\label{derv-est}
\big|{\partial^j\over\partial t^j}{\partial^i\over\partial x^i}u(t,x)\big|\le {C\over \sqrt
t}\|u_0'\|_{\gamma,1},
\end{equation}
where $\gamma=1$ if $2j+i=1$ or $2$; $\gamma=2$ if $2j+i=3$ or $4$,
and $\|g\|_{\gamma,1}=\sum_{i=1}^\gamma\|{\partial^ig\over\partial x^i}\|_1$.
\end{thm}
The proof of the theorem will be presented in the next subsection to
keep the presentation more transparent.

\subsection{The Numerical Method}

In this section we will construct an explicit one-to-one
transformation which transforms $Y^{(\alpha)}$ to a solution to a
stochastic differential equation without a local time which can
easily be discretized by a standard Euler-Maruyama scheme. Since the
transformation is one-to-one and explicit, we can take the inverse
transformation of this numerical solution to obtain a numerical
approximation for $Y^{(\alpha)}$. As a consequence of Theorem
\ref{thm:stoch-representation}, we can approximate $u(t,x)$ by
$E^xu_0(Y^{(\alpha)}(t))$ and compute the latter using the
Monte-Carlo simulation.

 To proceed, we denote
\begin{equation}\label{beta-defn}
\beta(x)=\lambda
 x\mathbf{1}_{(-\infty,0]}(x)+(1-\lambda)x\mathbf{1}_{(0,\infty)}(x).
\end{equation}
Then $\beta'_{-}(x)=\lambda
 \mathbf{1}_{(-\infty,0]}(x)+(1-\lambda)\mathbf{1}_{(0,\infty)}(x)$
 and
$\beta^{-1}(x)={x\over\lambda}
 \mathbf{1}_{(-\infty,0]}(x)+{x\over1-\lambda}\mathbf{1}_{(0,\infty)}(x)$.
 It follows that
\begin{equation}\label{alpha-defn}
\theta(x):=\beta_-'(x)\sqrt{D(x)}=\lambda\sqrt{D^-}
 \mathbf{1}_{(-\infty,0]}(x)+(1-\lambda)\sqrt{D^+}\mathbf{1}_{(0,\infty)}(x).
 \end{equation}
Since $\lambda \beta'_-(0^+)=(1-\lambda)\beta'_-(0^-)$ and $\beta\in
C^2(\mathbb{R}\setminus\{0\})\cap C(\mathbb{R})$, by virtue of
\eqref{martingale-prob},
\begin{align}
\beta\big(Y^{(\alpha)}(t)\big)&=\beta\big(Y^{(\alpha)}(0)\big)+\int_0^t\beta'_-\big(Y^{(\alpha)}(s)\big)\sqrt{D\big(Y^{(\alpha)}(s)\big)}dB(s)\notag\\
&=\beta\big(Y^{(\alpha)}(0)\big)+\int_0^t\theta\big(\beta\big(Y^{(\alpha)}(s)\big)\big)dB(s).\label{beta(Y)}
\end{align}
Denote $X(t)=\beta\big(Y^{(\alpha)}(t)\big)$, then \eqref{beta(Y)}
yields
\begin{equation}\label{X(t)-eqn}
X(t)=X(0)+\int_0^t\theta(X(s))dB(s).
\end{equation}
Let $\Delta=\Delta t={T\over M}$ be the step size. For $0\le k\le
M$, put $t_k=k\Delta t$. Let $\bar X^\Delta(t)$ be the
Euler-Maruyama approximation of $X(t)$,
\begin{equation}\label{barX^n-eqn}
\bar X^\Delta(t)=\bar X^\Delta(t_k)+\theta\big(\bar
X^\Delta(t_k)\big)\big(B(t)-B(t_k)\big),\quad \bar
X^\Delta(0)=\beta\big(Y^{(\alpha)}(0)\big).
\end{equation}
The numerical solution to \eqref{eq:diff} can be now obtained.
Define
\begin{equation}\label{barY^n-eqn}
\bar Y^\Delta(t)=\beta^{-1}\big(\bar X^\Delta(t)\big),\quad
u_\Delta(T,x)=E^xu_0(\bar Y^\Delta(T)).
\end{equation}
The convergence rate of the above numerical method is given in the
following theorem.
\begin{thm}\label{thm:num-r-est} For all initial condition $u_0\in \mathcal{W}^4$, all parameter
$0<\epsilon<1/2$ there exists a constant $C$ depending on $\epsilon$
such that for all $n$ large enough, and all $x_0\in{\mathbb R}$,
\begin{equation}\label{num-r-est}
\big|E^{x_0}u_0(Y^{(\alpha)}(T))-E^{x_0}u_0(\bar
Y^{\Delta}(T))\big|\le C\|u_0'\|_{1,1}\Delta
t^{(1-\epsilon)/2}+C\|u_0'\|_{1,1}\sqrt{\Delta
t}+C\|u_0'\|_{3,1}\Delta t^{1-\epsilon}.
\end{equation}
\end{thm}
Next, we can relax the transmission conditions of $u_0$ and
$\tilde{\mathcal{L}} u_0$ in the above theorem which are required in
the definition of $\mathcal{W}_4$.
\begin{thm}\label{thm:num-r-est-2} Let $u_0:\mathbb{R}\to\mathbb{R}$ be in the space
$$
\mathcal{W}=\Big\{g\in\mathcal{C}^4_b(\mathbb{R}\backslash\{0\}),
g^{(i)}\in L^1(\mathbb{R})\cap L^2(\mathbb{R}) \text{ for }
i=1,...,4\Big\}.
$$
Then for any parameter $0<\epsilon<1/2$ there exists a constant $C$
depending on $u_0$ and $\epsilon$ such that for all $n$ large
enough, and all $x_0\in{\mathbb R}$,
\begin{equation}\label{num-r-est2}
\big|u_\Delta(T,x_0)-u(T,x_0)\big|\le C \Delta t^{1/2-\epsilon}.
\end{equation}
\end{thm}

\subsection{Proofs} In this section we will gather the proofs of
\thmref{thm:derv-est}, \thmref{thm:num-r-est} and
\thmref{thm:num-r-est-2}.
\para{Proof of \thmref{thm:derv-est}}.
The proof will follow from a sequence of steps involving lemmas.

\underline{Step $1$:} Prove (i).\\
Let $p^{(\alpha)}(t,x,y)$ be the density function of the skew
Brownian motion $B^{(\alpha)}$, then according to \cite{walsh1},
\begin{equation}\label{skewBM density}
p^{(\alpha)}(t,x,y)=
\begin{cases}
{1\over\sqrt{2\pi t}}e^{-(y-x)^2\over2t}+{(2\alpha-1)\over\sqrt{2\pi
t}}e^{-(x+y)^2\over2t},&\text{if $x>0$, $y>0$,}\\
{1\over\sqrt{2\pi t}}e^{-(y-x)^2\over2t}-{(2\alpha-1)\over\sqrt{2\pi
t}}e^{-(x+y)^2\over2t},&\text{if $x<0$, $y<0$,}\\
{2\alpha\over\sqrt{2\pi t}}e^{-(y-x)^2\over2t},&\text{if $x\le0$, $y>0$,}\\
{2(1-\alpha)\over\sqrt{2\pi t}}e^{-(y-x)^2\over2t},&\text{if
$x\ge0$, $y<0$.}
\end{cases}
\end{equation}
Hence, it follows from \eqref{Y^alpha-defn} that $Y^{(\alpha)}(t)$
under $P^x$ has a density denoted by $q^{(\alpha)}(t,x,y)$ which
satisfies
\begin{equation}\label{Y^alpha-density}
q^{(\alpha)}(t,x,y)={1\over\sqrt{D(y)}}p^{(\alpha)}\Big(t,{x\over\sqrt{D(x)}},{y\over\sqrt{D(y)}}\Big).
\end{equation}
It is clear that \eqref{skewBM density} and \eqref{Y^alpha-density}
imply \eqref{density-est} and then \eqref{exptn-est}.

\underline{Step 2:} Estimate ${\partial u\over\partial t}(t,x)$. We
first prove the following lemma.
\begin{lem}\label{lem:u_t-est} There exists a positive constant $C$ such that for all
$t\in(0,T]$,
\begin{equation}\label{u_t-est}
\sup_{x\ne0}\big|{\partial u\over\partial t}(t,x)\big|\le{C\over\sqrt
t}\|u_0'\|_{1,1}.
\end{equation}
\end{lem}
\para{Proof.} Recall that for any $u_0\in
C^2_b(\mathbb{R}\setminus\{0\})\cap C(\mathbb{R})$ satisfying
$\lambda u_0'(0^+)=(1-\lambda)u_0'(0^-)$ \eqref{martingale-prob}
holds true. Hence, for all $x\in\mathbb{R}$ and $t>0$,
\begin{equation}\label{Dynkin-formula}
E^xu_0\big(Y^{(\alpha)}(t)\big)=g(x)+\int_0^tE^x\tilde{\mathcal{L}}u_0\big(Y^{(\alpha)}(s)\big)ds,
\end{equation}
where the operator $\tilde{\mathcal{L}}$ is defined as in
\eqref{operator L}. In addition, notice that
\begin{equation}\label{Y^al-eqn-localtime}
dY^{(\alpha)}(t)=\sqrt{\big(Y^{(\alpha)}(t)\big)}dB(t)+\Big({\sqrt{D^+}-\sqrt{D^-}\over2}+\sqrt{D^-}{2\alpha-1\over2\alpha}\Big)dl^{B^{(\alpha)},+}_t(0).
\end{equation}

Fix $x>0$. Denote
$\tau_0(Y^{(\alpha)})=\inf\{s>0:Y^{(\alpha)}(s)=0\}$ and $r_0^x(s)$
the density of $\tau_0(Y^{(\alpha)})\wedge T$ under $P^x$. Notice
that $\tau_0(Y^{(\alpha)})=\tau_0\big(x+\sqrt{D^+}B\big)$ where
$B(\cdot)$ is the standard Brownian motion. For all function $h$
such that $E^xh(Y^{(\alpha)})<\infty$ we have
\begin{align}
&E^xh(Y^{(\alpha)}(t))\notag\\
&=E^x\big[h(Y^{(\alpha)}(t))\mathbf{1}_{\{\tau_0\ge
t\}}\big]+E^x\big[h(Y^{(\alpha)}(t))\mathbf{1}_{\{\tau_0<
t\}}\big]\notag\\
&=E^x\big[h\big(x+\sqrt{D^+}B(t)\big)\mathbf{1}_{\{\tau_0\ge
t\}}\big]+\int_0^tE^0h\big(Y^{(\alpha)}(t-s)\big)r_0^x(s)ds\notag\\
&=E^xh\big(x+\sqrt{D^+}B(t)\big)-E^x\big[h\big(x+\sqrt{D^+}B(t)\big)\mathbf{1}_{\{\tau_0<
t\}}\big]+\int_0^tE^0h\big(Y^{(\alpha)}(t-s)\big)r_0^x(s)ds\notag\\
&=E^xh\big(x+\sqrt{D^+}B(t)\big)-\int_0^tE^0h\big(\sqrt{D^+}B(s)\big)r_0^x(t-s)ds+\int_0^tE^0h\big(Y^{(\alpha)}(s)\big)r_0^x(t-s)ds.\label{loc-time-aplctn}
\end{align}
For $x<0$ we have a similar identity. To proceed, we assume that
$x>0$. From \eqref{loc-time-aplctn} we can write
\begin{equation}\label{u(t,x)with-v}
u(t,x)=E^xu_0(Y^{(\alpha)}(t))=E^xu_0\big(x+\sqrt{D^+}B(t)\big)+v(t,x),
\end{equation}
where
\begin{align}
v(t,x)&=-\int_0^tE^0u_0\big(\sqrt{D^+}B(s)\big)r_0^x(t-s)ds+\int_0^tE^0u_0\big(Y^{(\alpha)}(s)\big)r_0^x(t-s)ds\notag\\
&=\int_0^t\int_0^{t-s}\Big[E^0\tilde{\mathcal{L}}u_0\big(Y^{(\alpha)}(\xi)\big)-E^0\tilde{\mathcal{L}}^+u_0\big(\sqrt{D^+}B(\xi)\big)\Big]d\xi
r_0^x(s)ds \label{v(t,x)-dfn}
\end{align}
and $\tilde{\mathcal{L}}^+u_0={D^+\over2}u_0''$. Since
\begin{equation}\label{dv/dt-eqn}
{\partial v\over\partial
t}(t,x)=\int_0^t\Big[E^0\tilde{\mathcal{L}}u_0\big(Y^{(\alpha)}(s)\big)-E^0\tilde{\mathcal{L}}^+u_0\big(\sqrt{D^+}B(s)\big)\Big]
r_0^x(t-s)ds,
\end{equation}
according to Lemma \ref{lem:r_0^x-int} we obtain,
\begin{align}
\Big|{\partial v\over\partial
t}(t,x)\Big|&\le\int_0^t\Big[\big|E^0\tilde{\mathcal{L}}u_0\big(Y^{(\alpha)}(s)\big)\big|+\big|E^0\tilde{\mathcal{L}}^+u_0\big(\sqrt{D^+}B(s)\big)\big|\Big]
r_0^x(t-s)ds\notag\\
&\le
C\big(\|\tilde{\mathcal{L}}u_0\|_1+\|\tilde{\mathcal{L}}^+u_0\|_1\big)\int_0^t{1\over
\sqrt s}r_0^x(t-s)ds\notag\\
&\le {C\over \sqrt
t}\big(\|\tilde{\mathcal{L}}u_0\|_1+\|\tilde{\mathcal{L}}^+u_0\|_1\big).\label{dv/dt-est}
\end{align}
Next we estimate ${\partial\over\partial t} E^xu_0(x+\sqrt{D^+}B(t))$. It is obvious that
the density $q^+(t,x,y)$ of $x+\sqrt{D^+}B(t)$ satisfies the
inequality $q^+(t,x,y)\le{C\over\sqrt t}\exp\{-{(y-x)^2\over\nu
t}\}$ for all $0\le t\le T$ for some constants $C,\nu$. It follows
from the equation
$$
{\partial\over\partial t}
E^xu_0(x+\sqrt{D^+}B(t))=E^x\tilde{\mathcal{L}}^+\big(u_0(x+\sqrt{D^+}B(t))\big)={D^+\over2}\int
u_0''(y)q^+(t,x,y)dy
$$
that
\begin{equation}\label{d_tE^xf-est}
\sup_{x\in\mathbb{R}}\Big|{\partial\over\partial t}
E^xu_0(x+\sqrt{D^+}B(t))\Big|\le{C\over\sqrt
t}\|\tilde{\mathcal{L}}^+u_0\|_1
\end{equation}
Combining \eqref{u(t,x)with-v}, \eqref{dv/dt-est} and
\eqref{d_tE^xf-est} we derive \eqref{u_t-est} as desired.\qed

\begin{lem}\label{lem:u_tt-est} There exists a positive constant $C$ such that for all
$t\in(0,T]$,
\begin{equation}\label{u_tt-est}
\sup_{x\ne0}\big|{\partial^2u\over\partial t^2}(t,x)\big|\le{C\over\sqrt
t}\|u_0'\|_{3,1}.
\end{equation}
\end{lem}
\para{Proof.} For $u_0\in\mathcal{W}^4$,
$\tilde{\mathcal{L}}u_0\in\mathcal{W}^2$. By virtue of
\eqref{u(t,x)with-v} and \eqref{dv/dt-eqn} we have
\begin{align}
{\partial^2 u\over\partial
t^2}(t,x)&={\partial^2\over\partial
t^2}E^xu_0(x+\sqrt{D^+}B(t))\notag\\
&\quad+\int_0^t{\partial\over\partial
t}\Big[E^0\tilde{\mathcal{L}}u_0\big(Y^{(\alpha)}(t-s)\big)-E^0\tilde{\mathcal{L}}^+u_0\big(\sqrt{D^+}B(t-s)\big)\Big]r_0^x(s)ds\notag\\
&={\partial\over\partial
t}E^x\tilde{\mathcal{L}}^+u_0(x+\sqrt{D^+}B(t))\notag\\
&\quad+\int_0^t\Big[E^0\tilde{\mathcal{L}}(\tilde{\mathcal{L}}u_0)\big(Y^{(\alpha)}(s)\big)-E^0\tilde{\mathcal{L}}^+(\tilde{\mathcal{L}}^+u_0)\big(\sqrt{D^+}B(s)\big)\Big]r_0^x(t-s)ds.\notag
\end{align}
Therefore, by \lemref{lem:u_t-est}, we obtain
\begin{align}
\Big|{\partial^2 u\over\partial t^2}(t,x)\Big|\le {C\over\sqrt
t}\|\tilde{\mathcal{L}}^+(\tilde{\mathcal{L}}^+u_0)\|_1+{C\over\sqrt
t}\Big(\|\tilde{\mathcal{L}}(\tilde{\mathcal{L}}u_0)\|_1+\|\tilde{\mathcal{L}}^+(\tilde{\mathcal{L}}^+u_0)\|_1\Big)={C\over\sqrt
t}\|u_0'\|_{3,1}.
\end{align}

\underline{Step 3:} Estimate ${\partial^iu\over\partial x^i} (t,x)$. We have the
following lemma.

\begin{lem}\label{lem:u_x-est} There exists a positive constant $C$ such that for all
$t\in(0,T]$,
\begin{equation}\label{u_x-est}
\sup_{x\ne0}\big|{\partial u\over\partial x}(t,x)\big|\le{C\over\sqrt
t}\|u_0'\|_{1,1}.
\end{equation}
\end{lem}
\para{Proof.} Since
${\partial \over\partial x}E^xu_0(x+\sqrt{D^+}B(t))=E^xu_0'(x+\sqrt{D^+}B(t))$, we
have
\begin{align}
\Big\|{\partial\over\partial
x}E^xu_0(x+\sqrt{D^+}B(t))\Big\|_{\infty}&=\big\|E^xu_0'(x+\sqrt{D^+}B(t))\big\|_{\infty}\notag\\
&=\big\|\int u_0'(y)q^+(t,x,y)dy\big\|_{\infty}\notag\\
&\le{C\over\sqrt t}\big\|\int u_0'(y)e^{-{(y-x)^2\over
D^+t^2}}dy\big\|_{\infty}\le{C\over\sqrt
t}\|u_0'\|_1.\label{d/dx-E^xf-est}
\end{align}

Let $H(s)=\int_0^s\big[E^0\tilde{\mathcal{L}}u_0(Y^{(\alpha)}(\xi))-E^0\tilde{\mathcal{L}}^+u_0(x+\sqrt{D^+}B(\xi))\big]d\xi$, then since
$H(0)=0$ and $H'(s)\le
C\big(\|\tilde{\mathcal{L}}u_0\|_1+\|\tilde{\mathcal{L}}^+u_0\|_1\big)={C_H\over
s^0}$ we have by \eqref{v(t,x)-dfn} and Lemma \ref{lem:r_0^x-est} that
$$
v(t,x)=\int_0^tH(s)r_0^x(t-s)ds
$$
and
\begin{equation}\label{|v_x|-inq}
\Big|{\partial v\over\partial x}(t,x)\Big|\le\tilde
C\big(\|\tilde{\mathcal{L}}^+u_0\|_1+\|\tilde{\mathcal{L}}u_0\|_1\big).
\end{equation}
In view of \eqref{u(t,x)with-v}, \eqref{d/dx-E^xf-est} and
\eqref{|v_x|-inq} we obtain \eqref{u_x-est}.\qed

By the similar way we can prove the estimates for
$|{\partial^j\over\partial t^j}{\partial^i\over\partial x^i}u(t,x)|$ for $2j+i\le 4$.\qed


\para{Proof of \thmref{thm:num-r-est}}. Denote $s_k=T-t_k$ for $0\le k\le M$. Since
$u(0,x)=u_0(x)$ and $u(T,x)=E^xu_0(Y^{(\alpha)}(T))$,
\begin{align*}
&u(0,\beta^{-1}(\bar X^\Delta(T))=u_0(\beta^{-1}(\bar X^\Delta(T))),\\
&u(T,x_0)=u(T,\bar
X^\Delta(0))=u(T,\beta^{-1}(Y^{(\alpha)}(0))=E^{x_0}u_0(Y^{(\alpha)}(T)).
\end{align*}
Therefore,
\begin{align}
\epsilon_T^{x_0}&=\Big|E^{x_0}u_0\big(Y^{(\alpha)}(T)\big)-E^{x_0}u_0\big(\bar
Y^{\Delta}(T)\big)\Big|=\Big|E^{x_0}u_0\big(\beta^{-1}(\bar
X(T))\big)-E^{x_0}u_0\big(\beta^{-1}(\bar X^\Delta(T))\big)\Big|\notag\\
&=\Big|E^{x_0}u\big(T,\beta^{-1}(\bar
X^\Delta(0))\big)-E^{x_0}u\big(0,\beta^{-1}(\bar X^\Delta(T))\big)\Big|\notag\\
&=\Big|\sum_{k=0}^{M-1}\Big[E^{x_0}u\big(T-t_k,\beta^{-1}(\bar
X^\Delta(t_k))\big)-E^{x_0}u\big(T-t_{k+1},\beta^{-1}(\bar
X^\Delta(t_{k+1}))\big)\Big]\Big|\notag\\
&\le\Big|\sum_{k=0}^{M-2}\Big[E^{x_0}u\big(s_k,\beta^{-1}(\bar
X^\Delta(t_k))\big)-E^{x_0}u\big(s_{k+1},\beta^{-1}(\bar
X^\Delta(t_{k+1}))\big)\Big]\Big|\notag\\
&\quad+\Big|E^{x_0}u\big(s_{M-1},\beta^{-1}(\bar
X^\Delta(t_{M-1}))\big)-E^{x_0}u\big(0,\beta^{-1}(\bar
X^\Delta(T))\big)\Big|.\label{eps-est}
\end{align}
To estimate the second term in \eqref{eps-est}, we use the fact that
$u(0,x)=u_0(x)$ and obtain
\begin{align*}
&\Big|E^{x_0}u\big(s_{M-1},\beta^{-1}(\bar
X^\Delta(t_{M-1}))\big)-E^{x_0}u\big(0,\beta^{-1}(\bar
X^\Delta(T))\big)\Big|\\
&\quad\le\Big|E^{x_0}u\big(s_{M-1},\beta^{-1}(\bar
X^\Delta(t_{M-1}))\big)-E^{x_0}u\big(0,\beta^{-1}(\bar
X^\Delta(t_{M-1}))\big)\Big|\\
&\quad\quad+\Big|E^{x_0}u_0\big(\beta^{-1}(\bar
X^\Delta(t_{M-1}))\big)- E^{x_0}u_0\big(\beta^{-1}(\bar
X^\Delta(T))\big)\Big|.
\end{align*}
Since $u_0''$ is in $L_1({\mathbb R})$, $u_0'$ is bounded and
$u_0\circ\beta^{-1}$ is Lipschitz. By virtue of the inequality
$\sup_{x\ne0}|{\partial u\over\partial t}(t,x)|\le {C\over\sqrt t}\|u_0'\|_{1,1}$ we
have
\begin{equation}\label{2nd-term}
\Big|E^{x_0}u\big(s_{M-1},\beta^{-1}(\bar
X^\Delta(t_{M-1}))\big)-E^{x_0}u\big(0,\beta^{-1}(\bar
X^\Delta(T))\big)\Big|\le C\|u_0'\|_{1,1}\sqrt{\Delta t}.
\end{equation}

It remains to estimate the first term in \eqref{eps-est}. To
proceed, we denote the time and space increments as follow
\begin{align*}
T_k&=u\big(s_k,\beta^{-1}(\bar
X^\Delta(t_k))\big)-u\big(s_{k+1},\beta^{-1}(\bar
X^\Delta(t_{k}))\big),\\
S_k&=u\big(s_{k+1},\beta^{-1}(\bar
X^\Delta(t_{k+1}))\big)-u\big(s_{k+1},\beta^{-1}(\bar
X^\Delta(t_{k}))\big).
\end{align*}
The first term in \eqref{eps-est} then can be rewritten as
$\big|\sum_{k=0}^{M-2}E^{x_0}(T_k-S_k)\big|$. The analysis of this
term will be divided into $4$ steps.

\underline{Step 1}: Estimate  for the time increment $T_k$: Since
$s_k-s_{k+1}=\Delta t$, by the definition of $T_k$ and Taylor
expansion we have
\begin{align*}
&\Big[u\big(s_k,\beta^{-1}(\bar
X^\Delta(t_k))\big)-u\big(s_{k+1},\beta^{-1}(\bar
X^\Delta(t_{k}))\big)\Big]1_{\{\bar X^\Delta(t_k)>0\}}\\
&=\Delta t{\partial u\over\partial t}\big(s_{k+1},\beta^{-1}(\bar
X^\Delta(t_{k}))\big)1_{\{\bar X^\Delta(t_k)>0\}}\\
&\quad+\Delta
t^2\int_{[0,1]^2}{\partial^2 u\over\partial t^2}\big(s_{k+1}+\tau_1\tau_2\Delta
t,\beta^{-1}(\bar
X^\Delta(t_{k}))\big)\tau_1d\tau_1d\tau_21_{\{\bar X^\Delta(t_k)>0\}}\\
&=T^+_k+R^+_k.
\end{align*}
Similarly,
\begin{align*}
&\Big[u\big(s_k,\beta^{-1}(\bar
X^\Delta(t_k))\big)-u\big(s_{k+1},\beta^{-1}(\bar
X^\Delta(t_{k}))\big)\Big]1_{\{\bar X^\Delta(t_k)<0\}}\\
&=\Delta t{\partial u\over\partial t}\big(s_{k+1},\beta^{-1}(\bar
X^\Delta(t_{k}))\big)1_{\{\bar X^\Delta(t_k)<0\}}\\
&\quad+\Delta
t^2\int_{[0,1]^2}{\partial^2 u\over\partial t^2}\big(s_{k+1}+\tau_1\tau_2\Delta
t,\beta^{-1}(\bar
X^\Delta(t_{k}))\big)\tau_1d\tau_1d\tau_21_{\{\bar X^\Delta(t_k)<0\}}\\
&=T^-_k+R^-_k.
\end{align*}
It follows from the above equations and the inequality
$\sup_{x\ne0}|{\partial^2 u\over\partial t^2}(t,x)|\le {C\over\sqrt t}\|u_0'\|_{3,1}$
that
$$
E^{x_0}\big|R^+_k+R^-_k\big|=E^{x_0}\Delta
t^2\Big|\int_{[0,1]^2}{\partial^2 u\over\partial t^2}\big(s_{k+1}+\tau_1\tau_2\Delta
t,\beta^{-1}(\bar X^\Delta(t_{k}))\big)\tau_1d\tau_1d\tau_2\Big|\le
{C\Delta t^2\over \sqrt{s_{k+1}}}\|u_0'\|_{3,1}.
$$
Therefore, we obtain
\begin{equation}\label{t-incr-est}
E^{x_0}T_k=\Delta
tE^{x_0}\Big[{\partial u\over\partial t}\big(s_{k+1},\beta^{-1}(\bar
X^\Delta(t_{k}))\big)\Big]+O\left({\Delta t^2\over
\sqrt{s_{k+1}}}\right).
\end{equation}

\underline{Step 2}: Estimate  for the space increment $S_k$: Let us
denote the following increments
\begin{align}
\triangle_{k+1}B&=B(t_{k+1})-B(t_k),\notag\\
\triangle_{k+1}\bar X^\Delta&=\theta\big(\bar
X^\Delta(t_k)\big)\triangle_{k+1}B,\notag\\
\tilde\triangle_{k+1}\bar Y^\Delta&={\triangle_{k+1}\bar
X^\Delta\over 1-\lambda}1_{\{\bar
X^\Delta(t_k)>0\}}+{\triangle_{k+1}\bar X^\Delta\over
\lambda}1_{\{\bar X^\Delta(t_k)<0\}},\label{tilde-triangle}
\end{align}
and events
\begin{align*}
\Omega_k^{++}&=\big\{\bar X^\Delta(t_k)>0,\bar
X^\Delta(t_{k+1})>0\big\},\quad\Omega_k^{+-}=\big\{\bar
X^\Delta(t_k)>0,\bar X^\Delta(t_{k+1})\le0\big\},\\
\Omega_k^{--}&=\big\{\bar X^\Delta(t_k)\le0,\bar
X^\Delta(t_{k+1})\le0\big\},\quad\Omega_k^{-+}=\big\{\bar
X^\Delta(t_k)\le0, \bar X^\Delta(t_{k+1})>0\big\}.
\end{align*}
Hence, by the definition of the function $\beta$, on
$\Omega^{++}_k$,
$$
\beta^{-1}\big(\bar X^\Delta(t_{k+1})\big)=\beta^{-1}\big(\bar
X^\Delta(t_{k})\big)+{\triangle_{k+1}\bar X^\Delta\over 1-\lambda}.
$$
This and Taylor expansion yield
\begin{align*}
&S_k1_{\Omega^{++}_k}\\
&={\triangle_{k+1}\bar X^\Delta\over
1-\lambda}{\partial u\over\partial x}\big(s_{k+1},\beta^{-1}(\bar
X^\Delta(t_{k}))\big)1_{\Omega^{++}_k}+{1\over2}{(\triangle_{k+1}\bar
X^\Delta)^2\over
(1-\lambda)^2}{\partial^2 u\over\partial x^2}\big(s_{k+1},\beta^{-1}(\bar
X^\Delta(t_{k}))\big)1_{\Omega^{++}_k}\\
&\quad+{1\over6}{(\triangle_{k+1}\bar X^\Delta)^3\over
(1-\lambda)^3}{\partial^3 u\over\partial x^3}\big(s_{k+1},\beta^{-1}(\bar
X^\Delta(t_{k}))\big)1_{\Omega^{++}_k}\\
&\quad+{(\triangle_{k+1}\bar X^\Delta)^4\over
(1-\lambda)^4}\int_{[0,1]^4}{\partial^4 u\over\partial x^4}\Big(s_{k+1},\beta^{-1}(\bar
X^\Delta(t_{k}))+\tau_1\tau_2\tau_3\tau_4{\triangle_{k+1}\bar
X^\Delta\over
1-\lambda}\Big)\tau_1\tau_2\tau_3d\tau_1...d\tau_41_{\Omega^{++}_k}\\
&=:S_k^{++1}+S_k^{++2}+S_k^{++3}+S_k^{++4}.
\end{align*}
Similarly,
\begin{align*}
&S_k1_{\Omega^{--}_k}\\
&={\triangle_{k+1}\bar X^\Delta\over
\lambda}{\partial u\over\partial x}\big(s_{k+1},\beta^{-1}(\bar
X^\Delta(t_{k}))\big)1_{\Omega^{--}_k}+{1\over2}{(\triangle_{k+1}\bar
X^\Delta)^2\over
\lambda^2}{\partial^2 u\over\partial x^2}\big(s_{k+1},\beta^{-1}(\bar
X^\Delta(t_{k}))\big)1_{\Omega^{--}_k}\\
&\quad+{1\over6}{(\triangle_{k+1}\bar X^\Delta)^3\over
\lambda^3}{\partial^3 u\over\partial x^3}\big(s_{k+1},\beta^{-1}(\bar
X^\Delta(t_{k}))\big)1_{\Omega^{--}_k}\\
&\quad+{(\triangle_{k+1}\bar X^\Delta)^4\over
\lambda^4}\int_{[0,1]^4}{\partial^4 u\over\partial x^4}\Big(s_{k+1},\beta^{-1}(\bar
X^\Delta(t_{k}))+\tau_1\tau_2\tau_3\tau_4{\triangle_{k+1}\bar
X^\Delta\over
\lambda}\Big)\tau_1\tau_2\tau_3d\tau_1...d\tau_41_{\Omega^{--}_k}\\
&=:S_k^{--1}+S_k^{--2}+S_k^{--3}+S_k^{--4}.
\end{align*}
Since
$\Omega^{++}_k\cup\Omega^{--}_k=\Omega-(\Omega^{+-}_k\cup\Omega^{-+}_k)$
and $\Omega^{+-}_k\cup\Omega^{-+}_k\in\sigma\{B(t):0\le t\le
t_{k+1}\}$, by \eqref{tilde-triangle} we get
\begin{align*}
&E^{x_0}(S^{++1}_k+S^{--1}_k)\\
&=E^{x_0}\Big[{\triangle_{k+1}\bar X^\Delta\over
1-\lambda}{\partial u\over\partial x}\big(s_{k+1},\beta^{-1}(\bar
X^\Delta(t_{k}))\big)1_{\Omega^{++}_k}+{\triangle_{k+1}\bar
X^\Delta\over \lambda}{\partial u\over\partial x}\big(s_{k+1},\beta^{-1}(\bar
X^\Delta(t_{k}))\big)1_{\Omega^{--}_k}\Big]\\
&\quad+E^{x_0}\big[\tilde\triangle_{k+1}\bar
Y^\Delta{\partial u\over\partial x}\big(s_{k+1},\beta^{-1}(\bar
X^\Delta(t_{k}))\big)1_{\{\Omega^{+-}_k\cup\Omega^{-+}_k\}}\big]\\
&\quad-E^{x_0}\big[\tilde\triangle_{k+1}\bar
Y^\Delta{\partial u\over\partial x}\big(s_{k+1},\beta^{-1}(\bar
X^\Delta(t_{k}))\big)1_{\{\Omega^{+-}_k\cup\Omega^{-+}_k\}}\big]\\
&=E^{x_0}\Big[\Big({\triangle_{k+1}\bar X^\Delta\over
1-\lambda}1_{\{\bar X^\Delta(t_k)>0\}}+{\triangle_{k+1}\bar
X^\Delta\over \lambda}1_{\{\bar
X^\Delta(t_k)<0\}}\Big){\partial u\over\partial x}\big(s_{k+1},\beta^{-1}(\bar
X^\Delta(t_{k}))\big)\Big]\\
&\quad-E^{x_0}\big[\tilde\triangle_{k+1}\bar
Y^\Delta{\partial u\over\partial x}\big(s_{k+1},\beta^{-1}(\bar
X^\Delta(t_{k}))\big)1_{\{\Omega^{+-}_k\cup\Omega^{-+}_k\}}\big]\\
&=-E^{x_0}\big[\tilde\triangle_{k+1}\bar
Y^\Delta{\partial u\over\partial x}\big(s_{k+1},\beta^{-1}(\bar
X^\Delta(t_{k}))\big)1_{\{\Omega^{+-}_k\cup\Omega^{-+}_k\}}\big].
\end{align*}
By the similar way and notice that
$E\big((\triangle_{k+1}B)^2\big|\,B(t),\, 0\le t\le t_k\big)=\Delta
t$, we obtain
\begin{align*}
&E^{x_0}(S^{++2}_k+S^{--2}_k)\\
&={1\over2}E^{x_0}\Big\{\Big[\Big({\triangle_{k+1}\bar X^\Delta\over
1-\lambda}\Big)^21_{\Omega^{++}_k}+\Big({\triangle_{k+1}\bar
X^\Delta\over
\lambda}\Big)^21_{\Omega^{--}_k}\Big]{\partial^2 u\over\partial x^2}\big(s_{k+1},\beta^{-1}(\bar
X^\Delta(t_{k}))\big)\Big\}\\
&={1\over2}E^{x_0}\Big\{\Big[\Big({\theta(\bar
X^\Delta(t_k))\triangle_{k+1}B\over 1-\lambda}\Big)^21_{\{\bar
X^\Delta(t_k)>0\}}+\Big({\theta(\bar
X^\Delta(t_k))\triangle_{k+1}B\over \lambda}\Big)^21_{\{\bar
X^\Delta(t_k)<0\}}\Big]\\
&\quad\times{\partial^2 u\over\partial x^2}\big(s_{k+1},\beta^{-1}(\bar
X^\Delta(t_{k}))\big)\Big\}-{1\over2}E^{x_0}\big[\big(\tilde\triangle_{k+1}\bar
Y^\Delta\big)^2{\partial^2 u\over\partial x^2}\big(s_{k+1},\beta^{-1}(\bar
X^\Delta(t_{k}))\big)1_{\{\Omega^{+-}_k\cup\Omega^{-+}_k\}}\big]\\
&={\Delta t\over2}E^{x_0}\Big\{\Big[D^+1_{\{\bar
X^\Delta(t_k)>0\}}+D^-1_{\{\bar
X^\Delta(t_k)<0\}}\Big]{\partial^2 u\over\partial x^2}\big(s_{k+1},\beta^{-1}(\bar
X^\Delta(t_{k}))\big)\Big\}\\
&\quad-{1\over2}E^{x_0}\big[\big(\tilde\triangle_{k+1}\bar
Y^\Delta\big)^2{\partial^2 u\over\partial x^2}\big(s_{k+1},\beta^{-1}(\bar
X^\Delta(t_{k}))\big)1_{\{\Omega^{+-}_k\cup\Omega^{-+}_k\}}\big]\\
&={\Delta t\over2}E^{x_0}\Big[D\big(\beta^{-1}(\bar
X^\Delta(t_k))\big){\partial^2 u\over\partial x^2}\big(s_{k+1},\beta^{-1}(\bar
X^\Delta(t_{k}))\big)\Big]\\
&\quad-{1\over2}E^{x_0}\big[\big(\tilde\triangle_{k+1}\bar
Y^\Delta\big)^2{\partial^2 u\over\partial x^2}\big(s_{k+1},\beta^{-1}(\bar
X^\Delta(t_{k}))\big)1_{\{\Omega^{+-}_k\cup\Omega^{-+}_k\}}\big].
\end{align*}
Since $E\big((\triangle_{k+1}B)^3\big|\,B(t),\, 0\le t\le
t_k\big)=0$,
\begin{align*}
&E^{x_0}(S^{++3}_k+S^{--3}_k)=-{1\over6}E^{x_0}\big[\big(\tilde\triangle_{k+1}\bar
Y^\Delta\big)^3{\partial^3 u\over\partial x^3}\big(s_{k+1},\beta^{-1}(\bar
X^\Delta(t_{k}))\big)1_{\{\Omega^{+-}_k\cup\Omega^{-+}_k\}}\big].
\end{align*}
Next, according to Theorem 2,
$$
E^{x_0}\big|S^{++4}_k+S^{--4}_k\big|\le{C\Delta t^2\over\sqrt{
s_{k+1}}}\|u_0'\|_{3,1}.
$$

Combining the calculations above, we arrive at
\begin{align}
E^{x_0}S_k&=E^{x_0}\tilde{\mathcal{L}}u\big(s_{k+1},\beta^{-1}(\bar
X^\Delta(t_{k}))\big)\Delta
t\notag\\
&\quad\quad+E^{x_0}\Big\{\Big[S_k-\tilde\triangle_{k+1}\bar
Y^\Delta{\partial u\over\partial x}\big(s_{k+1},\beta^{-1}(\bar
X^\Delta(t_{k}))\big)\Big]1_{\{\Omega^{+-}_k\cup\Omega^{-+}_k\}}\notag\\
&\quad\quad-{1\over2}\big(\tilde\triangle_{k+1}\bar
Y^\Delta\big)^2{\partial^2 u\over\partial x^2}\big(s_{k+1},\beta^{-1}(\bar
X^\Delta(t_{k}))\big)1_{\{\Omega^{+-}_k\cup\Omega^{-+}_k\}}\notag\\
&\quad\quad-{1\over6}\big(\tilde\triangle_{k+1}\bar
Y^\Delta\big)^3{\partial^3 u\over\partial x^3}\big(s_{k+1},\beta^{-1}(\bar
X^\Delta(t_{k}))\big)1_{\{\Omega^{+-}_k\cup\Omega^{-+}_k\}}\Big\}+O\Big({\Delta
t^2\over\sqrt{
s_{k+1}}}\Big)\notag\\
&=: E^{x_0}\tilde{\mathcal{L}}u\big(s_{k+1},\beta^{-1}(\bar
X^\Delta(t_{k}))\big)\Delta t+E^{x_0}\mathcal{R}_k+O\Big({\Delta
t^2\over\sqrt{ s_{k+1}}}\Big).\label{ES_k}
\end{align}
We now estimate the remaining term $E^{x_0}\mathcal{R}_k$.

\underline{Step 3}: Estimate $E^{x_0}\mathcal{R}_k$:

For any fixed $\epsilon\in(0,1/2)$, we will show that
\begin{align}
|E^{x_0}\mathcal{R}_k|&\le{C\Delta
t^{1-2\epsilon}\over\sqrt{s_{k+1}}}\|u_0'\|_{1,1}P^{x_0}\Big\{\big|\bar
X^\Delta(t_k)\big|\le \Delta t^{{1\over2}-\epsilon}\Big\}\notag\\
&\quad\quad\quad+{C\Delta
t^{{3\over2}-3\epsilon}\over\sqrt{s_{k+1}}}\|u_0'\|_{3,1}P^{x_0}\Big\{\big|\bar
X^\Delta(t_k)\big|\le \Delta
t^{{1\over2}-\epsilon}\Big\}.\label{ER_k}
\end{align}

Notice that we can rewrite $\Omega^{+-}_k$ as
\begin{align*}
\Omega^{+-}_k =&\,\,\Big\{\bar X^\Delta(t_k)\ge \Delta
t^{{1\over2}-\epsilon},\,\,\bar
X^\Delta(t_{k+1})\le0\Big\}\cup\Big\{0<\bar X^\Delta(t_k)\le \Delta
t^{{1\over2}-\epsilon},\,\,\bar X^\Delta(t_{k+1})\le-\Delta
t^{{1\over2}-\epsilon}\Big\}\\
&\cup\Big\{0<\bar X^\Delta(t_k)\le \Delta
t^{{1\over2}-\epsilon},\,\,-\Delta t^{{1\over2}-\epsilon}\le\bar
X^\Delta(t_{k+1})\le0\Big\}.
\end{align*}
Since $\bar X^\Delta(t_{k+1})=\bar X^\Delta(t_k)+\theta\big(\bar
X^\Delta(t_k)\big)B(\Delta t)$, it follows that
$$
P\Big\{\bar X^\Delta(t_k)\ge \Delta t^{{1\over2}-\epsilon},\,\,\bar
X^\Delta(t_{k+1})\le0\Big\}\le P\Big\{(1-\lambda)\sqrt{D^+}B(\Delta
t)\ge \Delta t^{{1\over2}-\epsilon}\Big\}\le C\exp\{-CM^\epsilon\}.
$$
Similarly,
$$
P\Big\{0<\bar X^\Delta(t_k)\le \Delta
t^{{1\over2}-\epsilon},\,\,\bar X^\Delta(t_{k+1})\le-\Delta
t^{{1\over2}-\epsilon}\Big\}\le C\exp\{-CM^\epsilon\}.
$$
We can proceed analogously on the event $\Omega^{-+}_k$. This leads
us to limit to consider the events
\begin{align*}
\hat\Omega^{+-}_k&=\Big\{0<\bar X^\Delta(t_k)\le \Delta
t^{{1\over2}-\epsilon},\,\,-\Delta t^{{1\over2}-\epsilon}\le\bar
X^\Delta(t_{k+1})\le0\Big\},\\
\hat\Omega^{-+}_k&=\Big\{-\Delta t^{{1\over2}-\epsilon}\le \bar
X^\Delta(t_k)<0,\,\,0\le\bar X^\Delta(t_{k+1})\le \Delta
t^{{1\over2}-\epsilon}\Big\}.
\end{align*}
Note that, by \eqref{tilde-triangle}, $\tilde \triangle_{k+1}\bar
Y^\Delta\le C\Delta t^{1/2-\epsilon}$ on these sets. Hence, we have
\begin{align*}
&\Big|E^{x_0}\Big[\big(\tilde\triangle_{k+1}\bar
Y^\Delta\big)^2{\partial^2 u\over\partial
x^2}\big(s_{k+1},\beta^{-1}(\bar
X^\Delta(t_{k}))\big)1_{\{\hat\Omega^{+-}_k\cup\hat\Omega^{-+}_k\}}\Big]\Big|\\
&\quad\le  {C\Delta
t^{1-2\epsilon}\over\sqrt{s_{k+1}}}\|u_0'\|_{1,1}P^{x_0}\Big\{\big|\bar
X^\Delta(t_k)\big|\le \Delta
t^{{1\over2}-\epsilon}\Big\},\\
&\Big|E^{x_0}\Big[\big(\tilde\triangle_{k+1}\bar
Y^\Delta\big)^3{\partial^3 u\over\partial x^3}\big(s_{k+1},\beta^{-1}(\bar
X^\Delta(t_{k}))\big)1_{\{\hat\Omega^{+-}_k\cup\hat\Omega^{-+}_k\}}\Big]\Big|\\
&\quad\le {C\Delta
t^{{3\over2}-3\epsilon}\over\sqrt{s_{k+1}}}\|u_0'\|_{3,1}P^{x_0}\Big\{\big|\bar
X^\Delta(t_k)\big|\le \Delta t^{{1\over2}-\epsilon}\Big\}.
\end{align*}
Therefore, it suffices to show that
\begin{align}
&\Big|E^{x_0}\Big[S_k-\tilde\triangle_{k+1}\bar
Y^\Delta{\partial u\over\partial x}\big(s_{k+1},\beta^{-1}(\bar
X^\Delta(t_{k}))\big)\Big]1_{\{\hat\Omega^{+-}_k\cup\hat\Omega^{-+}_k\}}\Big|\notag\\
&\quad\le{C\Delta
t^{1-2\epsilon}\over\sqrt{s_{k+1}}}\|u_0'\|_{1,1}P^{x_0}\Big\{\big|\bar
X^\Delta(t_k)\big|\le \Delta
t^{{1\over2}-\epsilon}\Big\}.\label{R_k-last-term}
\end{align}

\underline{Step 4}: Proof of \eqref{R_k-last-term}

Note that on the set $\hat\Omega^{+-}_k$, $\bar X^\Delta(t_{k})$ and
$\bar X^\Delta(t_{k+1})$ are both closed to $0$. In addition, $\bar
X^\Delta(t_{k})>0$ and $\bar X^\Delta(t_{k+1})<0$. Thus, we have
$$
\beta^{-1}\big(\bar X^\Delta(t_{k})\big)={\bar X^\Delta(t_{k})\over
1-\lambda},\quad \beta^{-1}\big(\bar X^\Delta(t_{k+1})\big)={\bar
X^\Delta(t_{k+1})\over \lambda}.
$$
Since $u(t,x)$ is continuous at $0$, we get
\begin{align*}
&E^{x_0}\Big[S_k-\tilde\triangle_{k+1}\bar
Y^\Delta{\partial u\over\partial x}\big(s_{k+1},\beta^{-1}(\bar
X^\Delta(t_{k}))\big)\Big]1_{\hat\Omega^{+-}_k}\\
&={1\over\lambda}E^{x_0}\Big[\bar
X^\Delta(t_{k+1}){\partial u\over\partial x}\big(s_{k+1},0^-\big)1_{\hat\Omega^{+-}_k}\Big]-{1\over1-\lambda}E^{x_0}\Big[\bar
X^\Delta(t_{k}){\partial u\over\partial x}\big(s_{k+1},0^+\big)1_{\hat\Omega^{+-}_k}\Big]\\
&\quad-E^{x_0}\Big[\tilde\triangle_{k+1}\bar
Y^\Delta{\partial u\over\partial x}\big(s_{k+1},0^+\big)1_{\hat\Omega^{+-}_k}\Big]\\
&\quad+E^{x_0}\Big[\Big(\big(\beta^{-1}\big(\bar
X^\Delta(t_{k+1})\big)\big)^2\int_{[0,1]^2}{\partial^2 u\over\partial x^2}\big(s_{k+1},\tau_1\tau_2\beta^{-1}\big(\bar
X^\Delta(t_{k+1})\big)\big)\tau_1d\tau_1d\tau_2\\
&\quad\quad-\big(\beta^{-1}\big(\bar
X^\Delta(t_{k})\big)\big)^2\int_{[0,1]^2}{\partial^2 u\over\partial x^2}\big(s_{k+1},\tau_1\tau_2\beta^{-1}\big(\bar
X^\Delta(t_{k})\big)\big)\tau_1d\tau_1d\tau_2\\
&\quad\quad-\tilde\triangle_{k+1}\bar Y^\Delta\beta^{-1}(\bar
X^\Delta(t_{k}))\int_0^1{\partial^2 u\over\partial x^2}\big(s_{k+1},\tau_1\beta^{-1}\big(\bar
X^\Delta(t_{k})\big)\big)d\tau_1\Big)1_{\hat\Omega^{+-}_k}\Big].
\end{align*}
On one hand, since $|\beta^{-1}\big(\bar X^\Delta(t_{k})\big)|$ and
$|\beta^{-1}\big(\bar X^\Delta(t_{k+1})\big)|\le C\Delta
t^{1/2-\epsilon}$ on $\hat\Omega^{+-}_k$, the absolute value of the
last expectation in the right-hand side can be bounded from above by
$$
{C\Delta
t^{1-2\epsilon}\over\sqrt{s_{k+1}}}\|u_0'\|_{1,1}P^{x_0}\Big(\big|\bar
X^\Delta(t_k)\big|\le \Delta t^{{1\over2}-\epsilon}\Big).
$$
On the other hand, by \eqref{tilde-triangle}, we can rewrite the sum
of the first three terms in the right hand side as
\begin{align}
&E^{x_0}\Big\{\Big[{\bar
X^\Delta(t_{k+1})\over\lambda}{\partial u\over\partial x}\big(s_{k+1},0^-\big)-{\bar
X^\Delta(t_{k})\over1-\lambda}{\partial u\over\partial x}\big(s_{k+1},0^+\big)\notag\\
&\quad\quad-{\bar
X^\Delta(t_{k+1})-\bar X^\Delta(t_{k})\over
1-\lambda}{\partial u\over\partial x}\big(s_{k+1},0^+\big)\Big]1_{\hat\Omega^{+-}_k}\Big\}\notag\\
=&E^{x_0}\Big\{\Big[{1\over\lambda}{\partial u\over\partial x}\big(s_{k+1},0^-\big)-{1\over
1-\lambda}{\partial u\over\partial x}\big(s_{k+1},0^+\big)\Big]\bar
X^\Delta(t_{k+1})1_{\hat\Omega^{+-}_k}\Big\}=0\notag
\end{align}
by the transmission condition. By the same way, we can proceed for
the set $\hat\Omega^{-+}_k$. Then \eqref{R_k-last-term} follows and
 we obtain \eqref{ER_k} as a consequence.

Next, combining \eqref{eps-est}-\eqref{t-incr-est}, \eqref{ES_k},
\eqref{ER_k} we arrive at
\begin{align*}
\epsilon_T^{x_0}&\le C\sum_{k=0}^{M-2}\Big({\Delta
t^{1-2\epsilon}\over\sqrt{s_{k+1}}}\|u_0'\|_{1,1}+{\Delta
t^{{3\over2}-3\epsilon}\over\sqrt{s_{k+1}}}\|u_0'\|_{3,1}\Big)
P^{x_0}\Big\{\big|\bar X^\Delta(t_k)\big|\le \Delta
t^{{1\over2}-\epsilon}\Big\}\\
&\quad\quad+C\|u_0'\|_{1,1}\Delta t^{1\over2}+C\|u_0'\|_{3,1}\Delta
t.
\end{align*}
By Theorem \ref{theorem:small-ball} it follows that there is a
constant $M_0$ such that for $M\ge M_0$, the right hand side in the
above inequality is bounded above by $C\|u_0'\|_{1,1}\Delta
t^{(1-\epsilon)/2}+ C\|u_0'\|_{1,1}\Delta
t^{1/2}+C\|u_0'\|_{3,1}\Delta t$. This proves the theorem.\qed

{\bf Proof of \thmref{thm:num-r-est-2}} Let $u_0$ be any function in
$\mathcal W$, and $0<\delta<1$ we will first approximate $u_0$ by a
function $u_\delta$ in $\mathcal{W}^4$ such that
$$
\begin{cases}
u_\delta(x)=u_0(x)&\text{for $|x|>2\delta$},\\
u_\delta(x)=u_0(0)&\text{for $-\delta\le x\le\delta$}
\end{cases}\quad
\text{and}\quad
\begin{cases}
u_\delta^{(i)}(2\delta)=u_0^{(i)}(2\delta),\\
u_\delta^{(i)}(-2\delta)=u_0^{(i)}(-2\delta),\\
u_\delta^{(i)}(-\delta)=u_\delta^{(i)}(\delta)=0
\end{cases}
\quad \text{for }1\le i\le4.
$$
For $\delta\le x\le2\delta$ denote
\begin{align*}
u_\delta(x)=&u_0(0)+\big(u_0(2\delta)-u_0(0)\big)p_0\big({x-\delta\over
\delta}\big)+\delta u_0^{(1)}(2\delta)p_1\big({x-\delta\over
\delta}\big)\\
&\quad\quad+\delta^2u_0^{(2)}(2\delta)p_2\big({x-\delta\over
\delta}\big)+\delta^3u_0^{(3)}(2\delta)p_3\big({x-\delta\over
\delta}\big)+\delta^4u_0^{(4)}(2\delta)p_4\big({x-\delta\over
\delta}\big),
\end{align*}
and for $-2\delta\le x\le-\delta$ denote
\begin{align*}
u_\delta(x)=&u_0(0)+\big(u_0(-2\delta)-u_0(0)\big)p_0\big(-{x+\delta\over
\delta}\big)-\delta u_0^{(1)}(-2\delta)p_1\big(-{x+\delta\over
\delta}\big)\\
&+\delta^2u_0^{(2)}(-2\delta)p_2\big(-{x+\delta\over
\delta}\big)-\delta^3u_0^{(3)}(-2\delta)p_3\big(-{x+\delta\over
\delta}\big)+\delta^4u_0^{(4)}(-2\delta)p_4\big(-{x+\delta\over
\delta}\big),
\end{align*}
where $p_j(x)$, $0\le j\le4$, are polynomials on $[0,1]$ satisfying
the following interpolation problem
$$
p^{(i)}_j(0)=0,\quad p^{(i)}_j(1)=\delta_{ij} \text{ for }0\le
i,j\le4,
$$
where $\delta_{ij}$ is the Kronecker symbol. We can choose
\begin{align*}
p_0(x)&=x^5(70x^4-315x^3+540x^2-420x+126),\\
p_1(x)&=x^5(1-x)(35x^3-120x^2+140x-56),\\
p_2(x)&={1\over2}x^5(1-x)^2(15x^2-35x+21),\\
p_3(x)&={1\over6}x^5(1-x)^3(5x-6),\\
p_4(x)&={1\over24}x^5(1-x)^4,
\end{align*}
which satisfy
$$
\Big\|p^{(i)}_j\Big({\cdot-\delta\over\delta}\Big)\Big\|_{L^1([\delta,2\delta])}+\Big\|p^{(i)}_j\Big(-{\cdot+\delta\over\delta}\Big)\Big\|_{L^1([-2\delta,-\delta])}\le
C\delta^{1-i},\quad \forall\,i=1,...,4
$$
and imply
$$
\|u_0-u_\delta\|_1=\int_{-2\delta}^{2\delta}|u_0(y)-u_0(0)+u_0(0)-u_\delta(y)|dy\le
C\delta^2.
$$
Similarly, there is a constant only depends on $u_0$ such that
\begin{equation}\label{f-derv-est}
\|u_0^{(i)}-u_\delta^{(i)}\|_1\le C\delta^{2-i}\quad\forall\,
i=1,...,4.
\end{equation}

Next, we will use the approximation $u_\delta$ of $u_0$ to estimate
the error $\epsilon^x_T$. We have
\begin{align}
\epsilon^x_T&=\big|E^xu_0\big(Y^{(\alpha)}(T)\big)-E^xu_0\big(\bar
Y^{\Delta}(T)\big)\big|\notag\\
&\le\big|E^xu_0\big(Y^{(\alpha)}(T)\big)-E^xu_\delta\big(
Y^{(\alpha)}(T)\big)\big|+\big|E^xu_\delta\big(Y^{(\alpha)}(T)\big)-E^xu_\delta\big(\bar
Y^{\Delta}(T)\big)\big|\notag\\
&\quad\quad+\big|E^xu_\delta\big(\bar
Y^{\Delta}(T)\big)-E^xu_0\big(\bar
Y^{\Delta}(T)\big)\big|\notag\\
&\le I_1(\delta)+I_2(\delta)+I_3(\delta).\label{I_123-eqn}
\end{align}
It follows from \eqref{density-est} that
\begin{equation}\label{I_1}
I_1(\delta)\le\int_{-\infty}^\infty|u_0(y)-u_\delta(y)|q^{(\alpha)}(T,x,y)dy\le
C\delta\int_{-2\delta}^{2\delta}q^{(\alpha)}(T,x,y)dy\le C\delta^2.
\end{equation}
By virtue of \thmref{thm:num-r-est} and \eqref{f-derv-est}, there
are constants $C$ depending only on $u$ such that
\begin{equation}\label{I_2}
I_2(\delta) \le C \|u_\delta'\|_{1,1}\Delta
t^{1/2-\epsilon}+C\|u_\delta'\|_{1,1}\Delta
t^{1/2}+\|u_\delta'\|_{3,1}\Delta t^{1-\epsilon}\le C\Delta
t^{1/2-\epsilon}+C\delta^{-2}\Delta t^{1-\epsilon}.
\end{equation}

To proceed, we need to estimate $I_3(\delta)$. Let $\chi$ be a
function in $C^\infty(\mathbb{R})$ such that
$$
\chi(x)\ge1\quad\forall\,
|x|\le1,\quad\text{and}\quad\chi^{(i)}(0)=0,\quad\forall\,i=1,...,4.
$$
Denote $\chi_\delta(x)=\chi({x\over2\delta})$ then
$\chi_\delta\ge\mathbf{1}_{[-2\delta,2\delta]}$,
$\text{supp}(\chi_\delta)=[-4\delta,4\delta]$, and it is clear that
$\chi,\chi_\delta\in\mathcal{W}^2$. In addition
$$
\|\chi_\delta'\|_{1,1}\le{C\over\delta}\quad\text{and}\quad\|\chi_\delta'\|_{3,1}\le{C\over\delta^3}.
$$
Thus, \thmref{thm:num-r-est} and \eqref{density-est} yield
\begin{align*}
P^x\big(|\bar Y^\Delta(T)|\le2\delta\big)&\le E^x\chi_\delta(\bar
Y^\Delta(T))\\
&\le \big|E^x\chi_\delta(\bar Y^\Delta(T))-E^x\chi_\delta(
Y^{(\alpha)}(T))\big|+E^x\chi_\delta(Y^{(\alpha)}(T))\\
&\le C\Delta t^{1-\epsilon\over2}\|\chi'\|_{1,1}+C\Delta
t^{1-\epsilon}\|\chi'\|_{3,1}+\int_{-4\delta}^{4\delta}\chi_\delta(y)q^{(\alpha)}(t,x,y)dy\\
&\le C{\Delta t^{1-\epsilon\over2}\over\delta}+C{\Delta
t^{1-\epsilon}\over\delta^3}+{C\over\sqrt T}\|\chi\|_\infty\delta
\end{align*}
and
\begin{equation}\label{I_3}
I_3(\delta)\le C\delta P^x\big(|\bar Y^\Delta(T)|\le2\delta\big)\le
C\Delta t^{1-\epsilon\over2}+C{\Delta
t^{1-\epsilon}\over\delta^2}+{C\over\sqrt T}\|\chi\|_\infty\delta^2.
\end{equation}
By choosing the optimal value of the type $\Delta t^\gamma$ of
$\delta$ (with $\gamma=1/4$), and combining
\eqref{I_123-eqn}-\eqref{I_3}, we obtain \eqref{num-r-est}.
 \qed

\section{Numerical Examples of Stochastic Method}
\label{sec:nums}
We again consider the initial profile given by \eqref{eq:example}.
Note that $u_0$ given in \eqref{eq:example} satisfies the conditions of Theorem \ref{thm:num-r-est-2}.
Numerical simulations are provided for
two values of $D^+=\{10,100\}$ while holding $D^-=1$.
We consider scenarios with
$\lambda=\{\lambda^*, \lambda^\#\}$.  Simulations of \eqref{eq:diff} with \eqref{eq:example} are shown in Figure \ref{fig:solution} above, along with the deterministic method.

For each choice of $\lambda$ and $D^+$ above, the error is computed between the stochastic numerical approximation and the expected value solution formula at specific points in space: $\{-1.5, 0, 2.5\}$.  To reduce the computational time, the largest stable time step was used, however the computations each involved over ten million sample paths.  
The absolute value of the error is plotted versus $dt:=h_n$ on a log-log plot in Figure \ref{fig:SDEerror}
demonstrating between zero and half order accuracy in each case, as predicted by Theorem \ref{thm:num-r-est-2}.

\begin{figure}[!t]
  \begin{center}
    \begin{tabular}{c}
      \mbox{\includegraphics[width = 0.5\textwidth,angle=0]{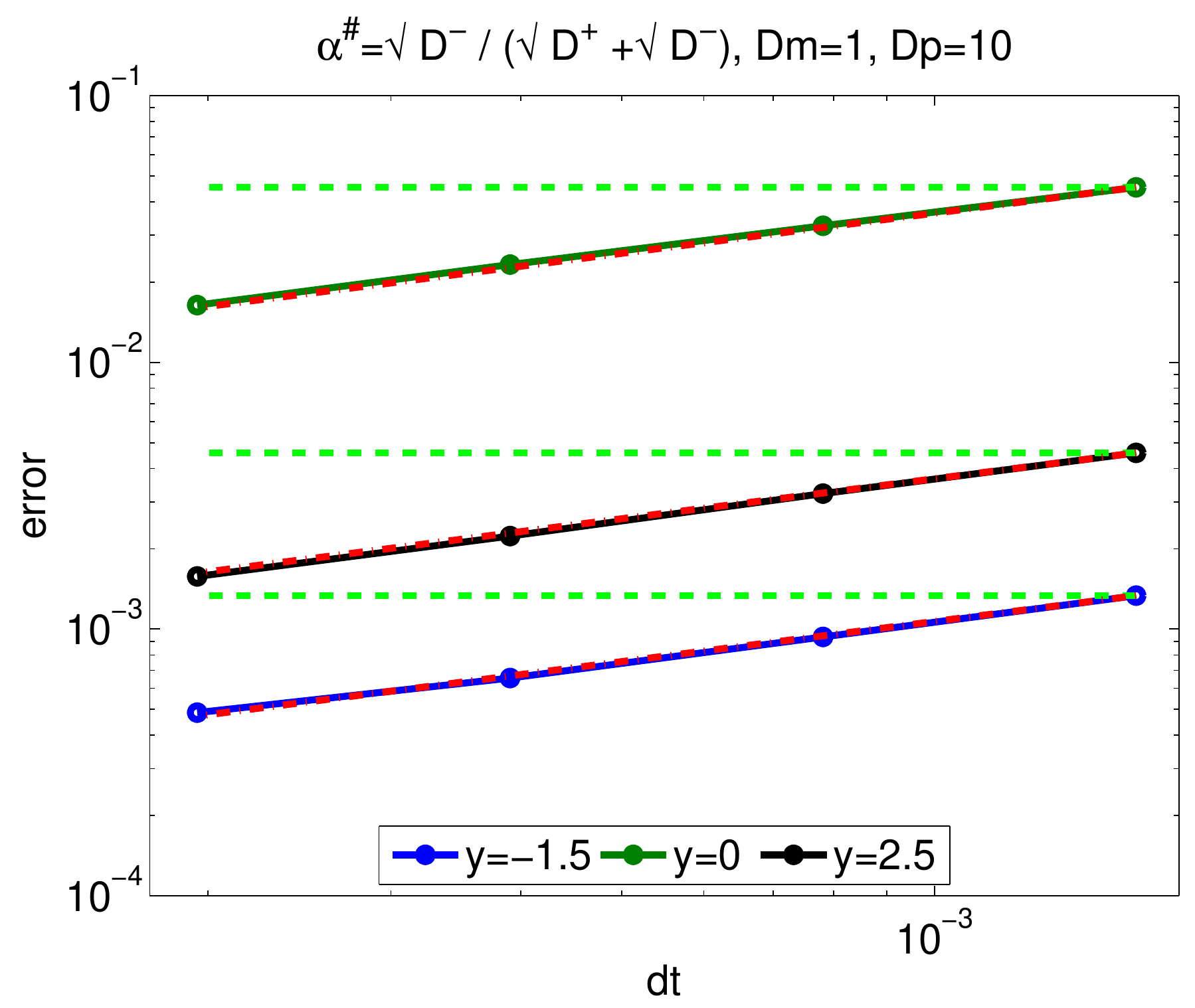}} \
      \mbox{\includegraphics[width=0.5\textwidth,angle=0]{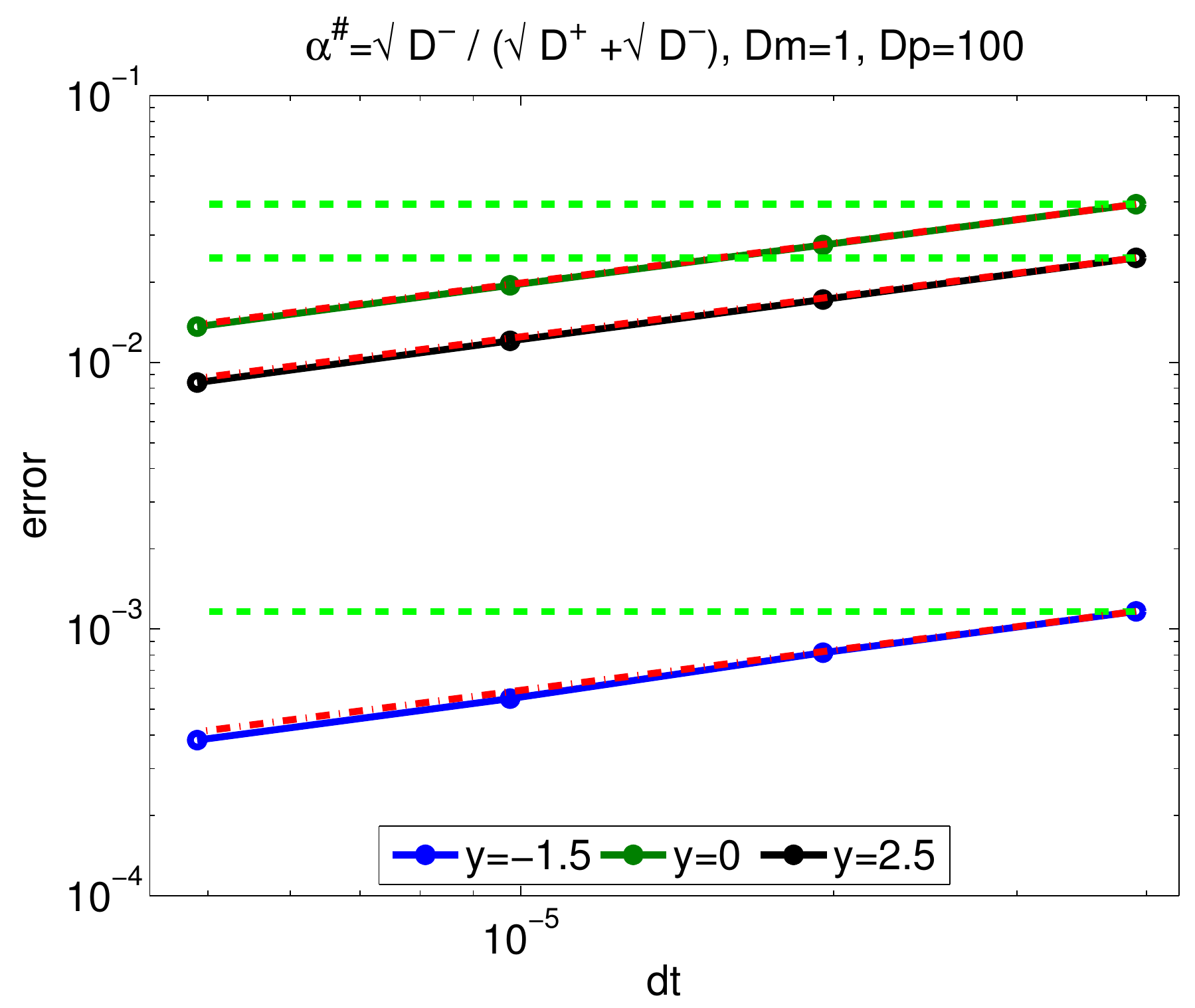}} \\
      \mbox{\includegraphics[width = 0.5\textwidth, angle=0]{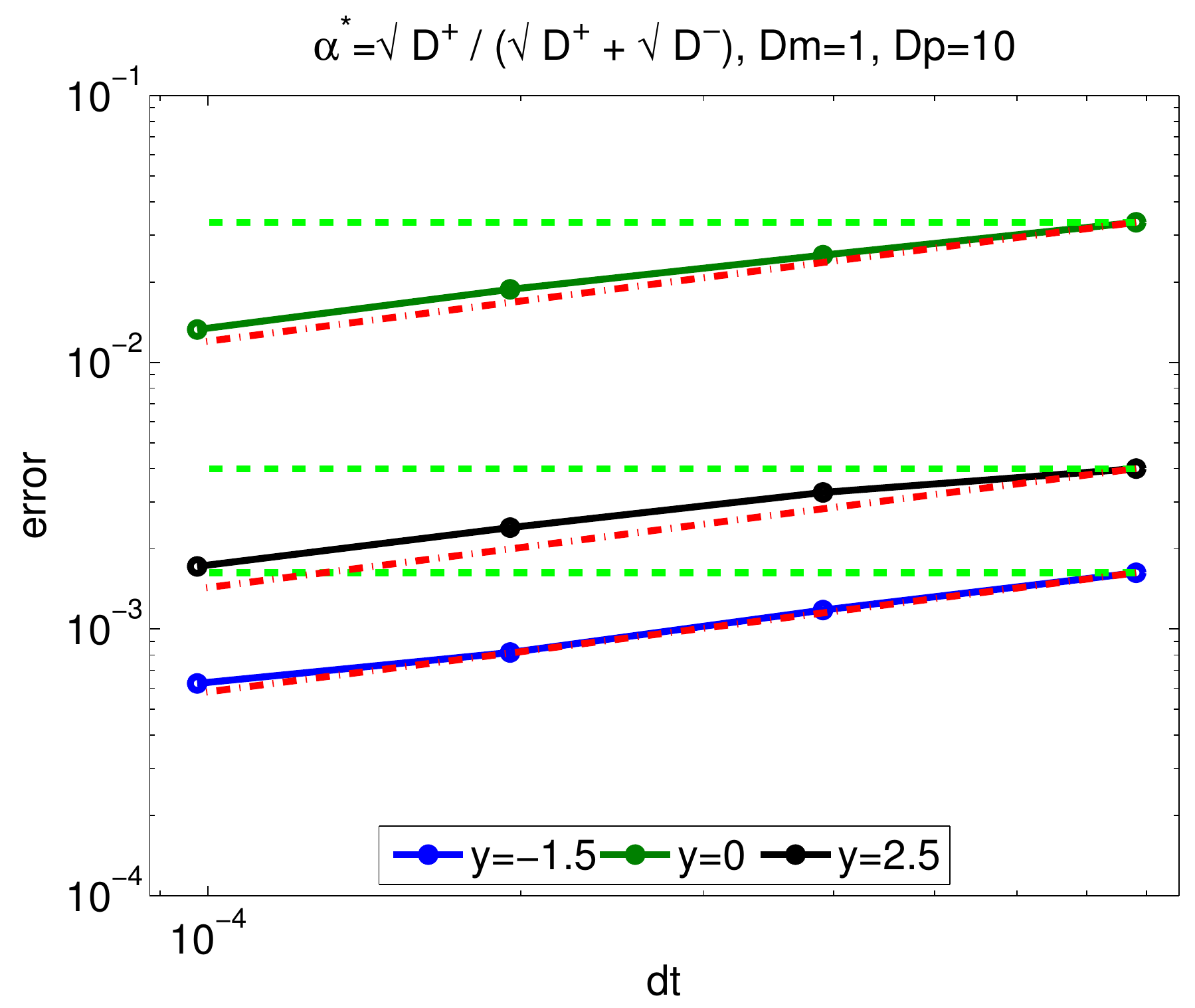}} \
      \mbox{\includegraphics[width=0.5\textwidth,angle=0]{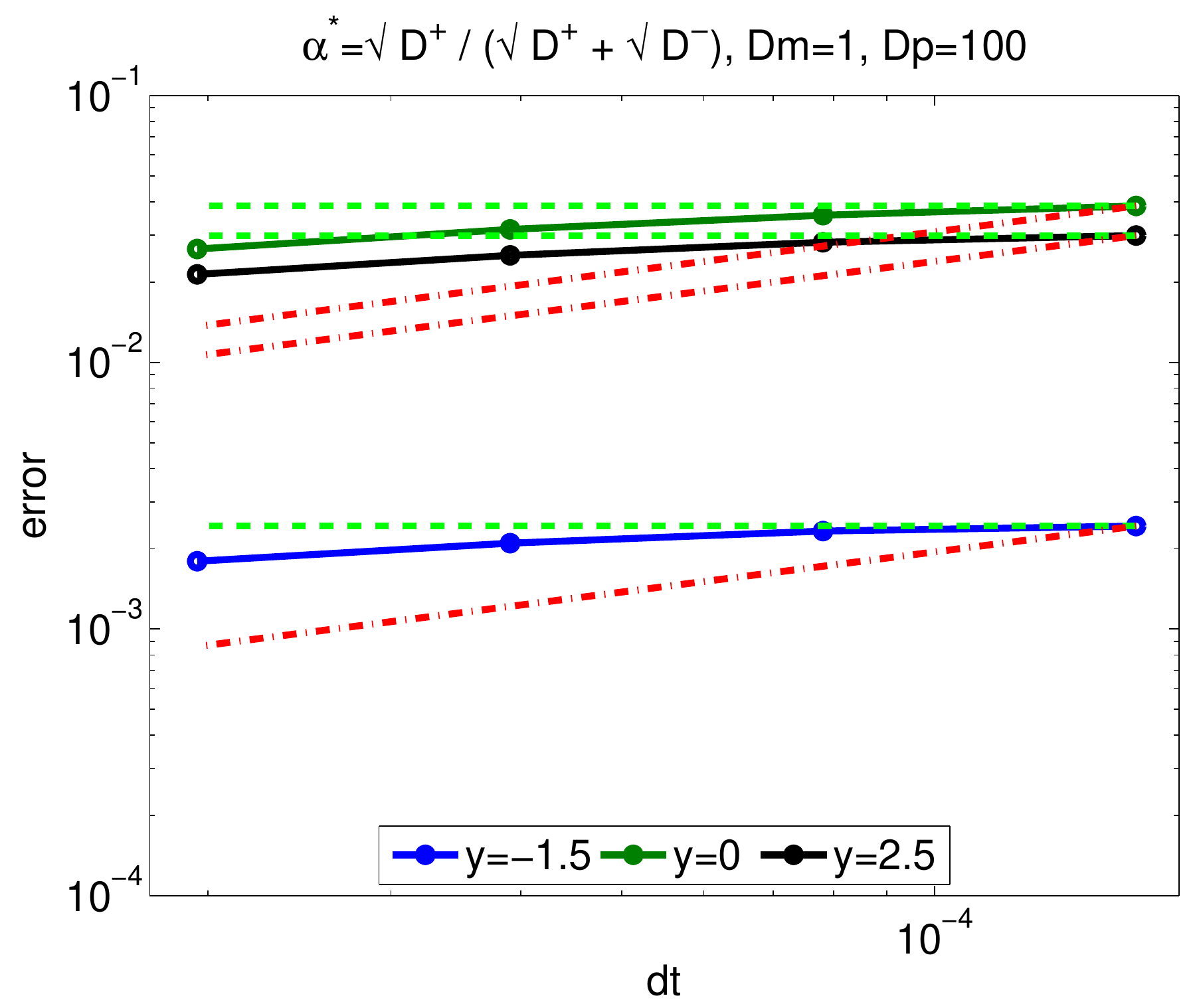}} \\
  \end{tabular}
  \end{center}\caption{Error for the SDE-BE method demonstrating half order convergence.}\label{fig:SDEerror}
\end{figure}


\section{Conclusions}

In this paper we have introduced a natural one parameter family of
possible interface conditions coupled to a diffusion problem, with
discontinuous diffusion coefficient, in one spatial dimension. We
then presented a reformation of the deterministic and the stochastic
models which naturally allow the application of numerical
discretization methods.  In particular, we chose to use the immersed
interface finite element method for the PDE.  We extended standard
energy estimates to show stability of the approach and derived error
estimates for the backward Euler case.  We demonstrated expected
convergence via numerical examples.  Finally, we introduced the
corresponding SDE and developed an Euler-Maruyama scheme applicable to any
one of the interface conditions.  We proved existence, uniqueness
and convergence of the numerical method under mild assumptions.
Again, the rates of convergence were verified by numerical examples.



\section{Appendix}
In this section we provide some properties of the first passage time
densities of one dimensional uniformly elliptic diffusion processes
which imply the estimates for the density $r_0^x(s)$ of the first
passage time before time $T$ at point $0$ of the process
$Y^{(\alpha)}$. In addition, we will present an estimate for the number of visits of small balls
by the Euler scheme.

The following Lemma is a combination of Theorem A.1
and Lemma A.5 in \cite{MartinezT12}.
\begin{lem}\label{lem:r_0^x-int}
Let $\gamma$ and $\mu$ be real valued functions such that $\gamma\in
C^{k+2}_b(\mathbb{R})$ and $\mu\in C^{k+1}_b(\mathbb{R})$ for some
non-negative integer $k$. Suppose that there is a positive constant
$\lambda$ such that $\gamma(x)>\lambda$ for all $x$ and $Z(t)$
satisfies
$$
Z(t)=Z_0+\int_0^t\mu(Z(s))ds+\int_0^t\gamma(Z(s))dB(s).
$$
a, If $T>0$ and $x\ne0$ then under $P^x$, the first passage time of
$Z(t)$ at point $0$ before time $T$, $\tau_0(Z)\wedge T$, has a
smooth density $r_0^x(s)$ which is of class $C^k((0,T])$.\\
b, In addition, if $k\ge2$ then for all $0\le\alpha<1$ there exists
a constant $C$ such that
$$
\int_0^t{1\over s^\alpha}r_0^x(t-s)ds\le{C\over
t^\alpha}\quad\text{for all $0\le t\le T$ and $x\ne0$}.
$$
\end{lem}

We also have the following estimate from \cite{MartinezT12} (See
Lemma A.6).
\begin{lem}\label{lem:r_0^x-est}
There exists a positive constant $\tilde C$ such that for
$0\le\alpha\le1$ and any function $H$ bounded on $[0,T]$,
continuously differentiable on $(0,T]$ satisfying
$$
H(0)=0,\quad |H'(s)|\le {C_H\over s^\alpha}\,\,\forall s\in(0,T]
$$
we have
$$
\Big|{\partial\over\partial x}\int_0^tr_0^x(t-s)H(s)ds\Big|\le
C_H\tilde C,\quad \Big|{\partial^2\over\partial
x^2}\int_0^tr_0^x(t-s)H(s)ds\Big|\le C_H\tilde C\Big(1+{1\over
t^\alpha}\Big),
$$
for all $t\in(0,T]$ and $x\ne0$.
\end{lem}

Next, let $W(\cdot)$ be a $m$-dimensional standard
Brownian motion on a filtered probability space $(\Omega,
\mathcal{F},\mathcal{F}_t, P)$. Assume that $b(\cdot)$ and
$\sigma(\cdot)$ are two progressive measurable processes taking
values in $\mathbb{R}^d$ and in the space of real $d\times m$
matrices, and $X(\cdot)$ is the $\mathbb{R}^d$-valued process
satisfying
\begin{equation}\label{X(t)-eqn2}
X(t)=X(0)+\int_0^tb(s)ds+\int_0^t\sigma(s)dW(s).
\end{equation}
Assume that\\
{\bf Assumption (A).} There exists a positive number $K\ge1$ such
that
\begin{equation}\label{b(t)-boundedness}
\forall\, t\ge0,\quad \|b(t)\|\le K\quad P-\text{a.s.}
\end{equation}
and
\begin{equation}\label{sigma(t)-boundedness}
\forall\, 0\le s\le t,\quad {1\over
K^2}\int_s^t\psi(s)ds\le\int_s^t\psi(s)\|\sigma\sigma^*\|ds\le
K^2\int_0^t\psi(s)ds
\end{equation}
for all positive locally integrable function $\psi$ on
$\mathbb{R}^+$.

\noindent{\bf Assumption (B).} $f$ is an increasing function in
$C^1([0,T),\mathbb{R}^+)$ such that $f^\alpha$ is integrable on
$[0,T)$ for all $1\le\alpha<2$. In addition, there exists
$1<\nu<1+\eta$ where $\eta={1\over 4K^4}$ such that
\begin{equation}\label{f-int-boundedness}
\int_0^Tf^{2\nu-1}(s)f'(s){(T-s)^{1+\eta}\over s^\eta}ds<\infty.
\end{equation}
Notice that \eqref{sigma(t)-boundedness} is satisfied if $\sigma$ is
a bounded continuous process. The Assumption (B) is satisfied if
$f(t)={1\over \sqrt{T-t}}$ and $\nu=1+{1\over8K^4}$.
We have the following estimate for the number of visits of small balls.

\begin{thm}[Theorem A.9 from \cite{MartinezT12}]\label{theorem:small-ball}
Assume (A) and (B). Let $X(\cdot)$ be as in \eqref{X(t)-eqn2}. Then
there exists a constant $C>0$ depending only on $\nu, K$ and $T$
such that for all $\xi\in\mathbb{R}^d$ and $0<\epsilon<1/2$, there
exists $h_0>0$ satisfying
\begin{equation}\label{small-ball-est}
\forall\, h\le h_0,\quad h\sum_{k=0}^{\lfloor
T/h\rfloor-1}f(kh)P\big(\|X(kh)-\xi\|\le h^{1/2-\epsilon}\big)\le
Ch^{1/2-\epsilon}.
\end{equation}
\end{thm}

\end{document}